\documentclass[reqno]{amsart}
\usepackage[margin=1.5in]{geometry}
\usepackage{palatino}
\usepackage{mathpazo}
\usepackage{amsmath, amssymb, amsthm}
\usepackage{tikz}
\usepackage{mathrsfs}
\usepackage[colorlinks=true, allcolors=purple]{hyperref}
\usepackage{tikz-cd}
\usepackage{enumitem}
\setlist[enumerate,1]{label=\textup{(\arabic*)}}
\setlist[enumerate,2]{label=\textup{(\alph*)}}

\newcommand{\bfem}[1]{\textbf{\textit{#1}}}
\newcommand{\LL}{\mathbb{L}}
\newcommand{\TT}{\mathbb{T}}
\newcommand{\Gm}{\mathbb{G}_m}
\newcommand{\cO}{\mathcal{O}}
\newcommand{\cL}{\mathcal{L}}
\newcommand{\K}{\mathbb{K}}
\newcommand{\Leg}{\mathrm{Leg}}
\newcommand{\Lag}{\mathrm{Lag}}
\newcommand{\Or}{\mathrm{Or}}
\newcommand{\DR}{\mathrm{DR}}
\newcommand{\Span}{\mathrm{Span}}
\newcommand{\QCoh}{\mathrm{QCoh}}
\newcommand{\dSt}{\mathrm{dSt}}
\newcommand{\PreSymp}{\mathrm{PreSymp}}
\newcommand{\POr}{\mathrm{PreOr}}
\newcommand{\Cospan}{\mathrm{Cospan}}
\newcommand{\Spec}{\operatorname{Spec}}
\newcommand{\Fc}{\mathcal{F}_{c}}

\newtheorem{theorem}{Theorem}[section]
\newtheorem{proposition}[theorem]{Proposition}
\newtheorem{lemma}[theorem]{Lemma}
\newtheorem{corollary}[theorem]{Corollary}
\newtheorem{thmx}{Theorem}

\numberwithin{equation}{section}
\theoremstyle{definition}
\newtheorem{definition}[theorem]{Definition}
\newtheorem{remark}[theorem]{Remark}
\newtheorem{example}[theorem]{Example}

\makeatletter
\g@addto@macro{\endabstract}{\@setabstract}
\newcommand{\authorfootnotes}{\renewcommand\thefootnote{\@fnsymbol\c@footnote}}%
\makeatother

\begin{document}

\title{A Derived Legendrian Category for Shifted Contact Stacks}
\author{Efe \.{I}zbudak}
\author{Kadri \.{I}lker Berktav}

\begin{center}
  \LARGE
  A Derived Legendrian Category for Shifted Contact Stacks \par \bigskip

  \normalsize
  \authorfootnotes
  Efe
  \.{I}zbudak\footnote{efe.izbudak@studium.uni-hamburg.de}\textsuperscript{1},
  Kadri \.{I}lker
  Berktav\footnote{berktav@metu.edu.tr}\textsuperscript{2} \par \bigskip

  \textsuperscript{1}Department of Mathematics, Universit\"{a}t Hamburg\par
  \textsuperscript{2}Department of Mathematics, METU\par \bigskip

\end{center}

\begin{abstract}
  For every integer $s$ and every $n \geq 1$, we construct a
  symmetric monoidal $(\infty,n)$-category $\Leg_n^s$ whose objects
  are $s$-shifted contact derived Artin stacks and whose
  $i$-morphisms are $i$-fold Legendrian correspondences. The
  construction is based on the fact that an $s$-shifted contact
  structure on a derived Artin stack $X$ is an $s$-shifted symplectic
  structure on $X$ relative to $B\Gm$, twisted by the weight-one line
  bundle, and that Legendrian structures are the corresponding
  relative Lagrangian structures. This lets us obtain the
  $(\infty,n)$-categories $\Leg_n^s$ as twisted forms over $B\Gm$ of
  the $(\infty,n)$-categories of iterated Lagrangian correspondences
  of Calaque, Haugseng and Scheimbauer, with the fiber product over
  $B\Gm$ as tensor product.

  For an $s$-shifted contact stack $X$, we then obtain the derived
  Legendrian category $\mathcal{F}_{c}(X)$. Its objects are
  Legendrian morphisms into $X$, and its mapping categories are
  categories of Legendrians in derived Legendrian intersections,
  which carry $(s-1)$-shifted contact structures. The AKSZ
  construction gives oriented extended topological field theories
  with values in $\Leg_n^s$, as well as shifted contact structures on
  the quotient mapping stacks $[\mathrm{Map}(Y,\widetilde X)/\Gm]$.
  Our definitions of contact and Legendrian structures agree with
  those of Berktav where both are defined, for every shift.
\end{abstract}
{\small\tableofcontents}

\section{Introduction}\label{sec:intro}

Shifted symplectic derived stacks and their Lagrangian
correspondences form higher categories. For instance, Amorim and
Ben-Bassat constructed a bicategory of Lagrangians in a fixed shifted
symplectic derived stack and, as a special case, a bicategory of
$n$-shifted symplectic derived stacks and Lagrangian correspondences
\cite[Theorem~4.11 and Corollary~4.12]{AmorimBassat}. Later, Calaque,
Haugseng and Scheimbauer constructed, for every derived stack $S$,
symmetric monoidal $(\infty,n)$-categories $\Lag_n^{S,s}$ of
$s$-shifted symplectic derived stacks over $S$ and iterated
Lagrangian correspondences, whose homotopy bicategory for $S =
\mathrm{pt}$ and $n = 2$ is the second bicategory of Amorim and
Ben-Bassat \cite[Definition~2.9.7, Corollary~2.10.8 and
\S1.3]{CHS2025}. In the same work, they also used the AKSZ
construction to obtain oriented extended topological field theories
with values in these $(\infty,n)$-categories \cite[Corollary~5.6.3]{CHS2025}.

In this paper, we construct the corresponding structures for shifted
contact derived Artin stacks. To do so, we work throughout with the
principal $\Gm$-bundle of the contact line bundle. Recall that a
\emph{line bundle $\cL$ on a derived stack $X$} is a morphism $X \to
B\Gm$, and that its principal $\Gm$-bundle is $\widetilde X = X
\times_{B\Gm} \mathrm{pt}$. We then define an \emph{$s$-shifted
contact structure on $X$} with contact line bundle $\cL$ to be a
weight-one $s$-shifted symplectic structure on $\widetilde X$
(Definition~\ref{def:contact}), and a \emph{Legendrian structure} on
a morphism $g \colon Y \to X$ to be a weight-one Lagrangian structure
on its lift $\widetilde Y \to \widetilde X$ (Definition~\ref{def:legendrian}).

Equivalently, a contact structure is an $s$-shifted symplectic
structure on $X$ relative to $B\Gm$, with coefficients twisted by the
weight-one line bundle on $B\Gm$ (Proposition~\ref{prop:relative}).
Moreover, when $\cL$ is trivial on an \'etale cover of $X$, these
notions agree with those of Berktav \cite[Definition~3.6]{kib1},
\cite[Definition~2.17]{kib2}, \cite[Definition~3.15]{kib3}, which are
also the ones used in \cite[Definitions~2.38
and~2.39]{IzbudakBerktav2026} (Theorem~\ref{thm:D}). Our definitions,
however, do not need this condition, which fails for the unit $B\Gm$
of the constructions below.

With these definitions in place, our constructions rest on the fact
that the relative de Rham complex of
\cite[Definition~B.11.4]{CHS2025} is linear over the structure sheaf
of the base. As a result, the constructions of \cite{CHS2025} can be
carried out over $B\Gm$ with twisted forms in place of forms
(Theorem~\ref{thm:transfer}), and the fiber product over $B\Gm$ then
provides the tensor product. In brief, we obtain the following result.

\begin{thmx}[cf.~Theorem~\ref{thm:Leg}]\label{thm:A}
  For every $s \in \mathbb{Z}$ and $n \geq 1$, there is a symmetric
  monoidal $(\infty,n)$-category $\Leg_n^s$ whose objects are
  $s$-shifted contact derived Artin stacks and whose $i$-morphisms
  are $i$-fold Legendrian correspondences. The tensor product is the
  contact product
  \[ X \circledast Y := X \times_{B\Gm} Y \simeq [(\widetilde X
  \times \widetilde Y)/\Gm], \]
  with unit $B\Gm$. The $(\infty,n)$-category $\Leg_n^s$ has adjoints
  for $1$-morphisms, and all its objects are fully dualizable.
\end{thmx}

In particular, for an $s$-shifted contact derived Artin stack $X$,
the Legendrian morphisms into $X$ are the $1$-morphisms from the unit
$B\Gm$ to $X$ in $\Leg_{n+1}^{s}$. These form the derived Legendrian
category of $X$, for which we have the following result.

\begin{thmx}[cf.~Theorem~\ref{thm:Fc} and
  Corollary~\ref{cor:intersection}]\label{thm:B}
  Let $X$ be an $s$-shifted contact derived Artin stack and $n \geq
  1$, and let $\Fc(X) := \Leg_{n+1}^{s}(B\Gm, X)$.
\begin{enumerate}\itemsep=0.2cm
    \item The objects of the $(\infty,n)$-category $\Fc(X)$ are the
      Legendrian morphisms $L \to X$.

    \item For Legendrian morphisms $L_1, L_2 \to X$, the derived
      intersection $L_1 \times_X L_2$ carries a canonical
      $(s-1)$-shifted contact structure, whose symplectification is
      the derived intersection of the Lagrangian lifts $\widetilde
      L_1, \widetilde L_2 \to \widetilde X$. The mapping
      $(\infty,n-1)$-category $\Fc(X)(L_1, L_2)$ is the
      $(\infty,n-1)$-category of Legendrians in $L_2 \times_X L_1$.
    \item Composition in $\Fc(X)$ is induced by the
      $(\infty,n+1)$-category $\Leg_{n+1}^s$, and on $1$-morphisms it
      is given by fiber products, i.e., Legendrians $N_{01} \to L_1
      \times_X L_0$ and $N_{12} \to L_2 \times_X L_1$ compose to
      $N_{01} \times_{L_1} N_{12} \to L_2 \times_X L_0$.
  \end{enumerate}
\end{thmx}

Together with Theorem~\ref{thm:D} below, part (2) recovers the
Legendrian intersection theorem of
\cite[Theorem~4.1]{IzbudakBerktav2026}
(Corollary~\ref{cor:intersection}). In addition, the adjoints of
Theorem~\ref{thm:A} provide further structure. More precisely, since
$1$-morphisms in $\Leg_{2}^{s}$ have adjoints, the self-intersection
$L \times_X L$ of a Legendrian morphism $L \to X$ is an associative
algebra in the monoidal $\infty$-category of $(s-1)$-shifted contact
stacks and Legendrian correspondences, with the contact product as
tensor product (Corollary~\ref{cor:monad}). Similarly, the AKSZ
construction of \cite[Corollary~3.4.3]{CHS2025} carries over to
contact stacks, as follows.

\begin{thmx}[cf.~Theorem~\ref{thm:aksz}]\label{thm:C}
  Let $X$ be an $s$-shifted contact derived Artin stack. For every $n
  \geq 1$ the AKSZ construction defines an oriented extended
  topological field theory
  \[
    \begin{tikzcd}
      \mathrm{Bord}^{\mathrm{or}}_{0,n} \arrow{r}{\mathcal{Z}_X} & \Leg_n^{s},
    \end{tikzcd}
  \]
  where $\mathrm{Bord}^{\mathrm{or}}_{0,n}$ is the oriented bordism
  $(\infty,n)$-category of \cite[Variant~4.1.9 and
  Definition~4.1.11]{CHS2025}. Moreover, let $Y$ be an $\cO$-compact
  derived stack with an $\cO$-orientation of degree $d$, such as the
  Betti stack of a closed oriented $d$-manifold, and suppose that
  $\mathrm{Map}(Y, \widetilde X)$ is a derived Artin stack. Then the
  AKSZ construction equips the quotient mapping stack
  $[\mathrm{Map}(Y, \widetilde X)/\Gm]$, which is the relative
  mapping stack from $Y \times B\Gm$ to $X$ over $B\Gm$, with an
  $(s-d)$-shifted contact structure.
\end{thmx}

\noindent Note that the last assertion recovers
\cite[Theorem~A]{IzbudakBerktav3} (Remark~\ref{rem:aksz-recovers}).
Finally, we compare our definitions with those of Berktav, using the
Euler vector field of the $\Gm$-action on $\widetilde X$.

\begin{thmx}[cf.~Proposition~\ref{prop:comparison} and
  Theorem~\ref{thm:legendrian-comparison}]\label{thm:D}
  Let $X$ be a derived Artin stack, $\cL$ a line bundle on $X$ with
  principal $\Gm$-bundle $p \colon \widetilde X \to X$, and $s \in
  \mathbb{Z}$. Then pulling back to $\widetilde X$ and applying the
  derived de Rham differential gives an equivalence from the space of
  $\cL$-twisted $s$-shifted $1$-forms on $X$ to the space of
  weight-one closed $s$-shifted $2$-forms on $\widetilde X$, whose
  inverse is contraction with the Euler vector field of the
  $\Gm$-action. Suppose also that $\cL$ is trivial on an \'etale cover of $X$.
\begin{enumerate}\itemsep=0.1cm
    \item A twisted $1$-form is a contact form in the sense of
      \cite[Definition~2.17]{kib2} and
      \cite[Definition~2.38]{IzbudakBerktav2026} if and only if the
      corresponding weight-one closed $2$-form is symplectic.
    \item Let $g \colon Y \to X$ be a morphism and $h$ a
      null-homotopy of the pullback of a contact form. Then $h$ is a
      Legendrian structure in the sense of
      \cite[Definition~3.15]{kib3} and
      \cite[Definition~2.39]{IzbudakBerktav2026} if and only if the
      corresponding weight-one isotropic structure on $\widetilde Y
      \to \widetilde X$ is Lagrangian.
  \end{enumerate}
\end{thmx}

In particular, every $s$-shifted contact derived Artin stack in the
sense of Berktav has a symplectification, for every $s \in
\mathbb{Z}$, which for $s < 0$ was proved in \cite[Theorem~4.7]{kib1}
for derived schemes and in \cite[Theorem~3.9]{kib2} for derived Artin stacks.

In Section~\ref{sec:applications}, we then apply the construction to
examples. First, for a derived Artin stack $Y$, the projectivized
shifted cotangent stack $[T^*[s]Y/\Gm]$ is $s$-shifted contact, with
the zero section included (Example~\ref{ex:PT}). Jet spaces provide a
second example. Let $V$ be a smooth affine $\K$-scheme and let $f_i,
f_j$ be functions on $V$. Then, in the derived Legendrian category of
the jet space $J^1(V)$, the morphisms from the $1$-jet of $f_i$ to
the $1$-jet of $f_j$ are the Legendrians in the derived discriminant
locus of $f_j - f_i$. These morphisms compose by fiber products over
$V$, and the triple intersections of $1$-jets are the apexes of the
whiskered counits of the adjunctions of Theorem~\ref{thm:A}
(Proposition~\ref{prop:discriminant}). Recall that derived
discriminant loci are the local models of $(-1)$-shifted contact
derived schemes \cite[Corollary~3.9]{Izbudak_Darboux}.

\medskip
\paragraph{\bfem{Organization of the paper.}}
Section~\ref{sec:recollections} recalls the results of \cite{PTVV}
and \cite{CHS2025} that we use. In Section~\ref{sec:contact} we
define weights, shifted contact structures and Legendrian morphisms
via symplectifications, and prove Theorem~\ref{thm:D}.
Section~\ref{sec:twisted} introduces twisted relative forms and uses
them to identify contact geometry with twisted symplectic geometry
over $B\Gm$. With this, Section~\ref{sec:Leg} constructs $\Leg_n^s$
and proves Theorem~\ref{thm:A}, and Sections~\ref{sec:Fc}
and~\ref{sec:aksz} treat derived Legendrian categories and the AKSZ
construction and prove Theorems~\ref{thm:B} and~\ref{thm:C}. Finally,
Section~\ref{sec:applications} contains the applications, and
Appendix~\ref{app:chs} recalls the results of \cite{CHS2025} on
diagrams of quasi-coherent sheaves and on the composition of
Lagrangian correspondences, proves the twisted form of the first, and
collects the facts about derived stacks from \cite{CHS2025} that are
used in the proofs.

\medskip
\paragraph{\bfem{Conventions and notations.}}
Throughout the paper, $\K$ is an algebraically closed field of
characteristic zero. All cdgas are graded in nonpositive degrees and
defined over $\K$. Classical $\K$-schemes are locally of finite type,
and derived $\K$-stacks are locally finitely presented. We work with
$\infty$-categories as in \cite[\S1.4]{CHS2025}. We write
$\mathrm{pt} = \Spec \K$, and $q \colon \mathrm{pt} \to B\Gm$ for the
universal $\Gm$-torsor. For a derived stack $S$ we write $\dSt_S$ for
the $\infty$-category of derived stacks over $S$.

We also use the notation of \cite[Notation~2.1.1]{CHS2025}. In
particular, $\dSt^{\mathrm{Art}}_S \subseteq \dSt_S$ is the full
subcategory of Artin $S$-stacks, $\QCoh(X)$ is the symmetric monoidal
stable $\infty$-category of quasi-coherent sheaves on $X$, and
$\LL_{X/S}$ and $\TT_{X/S} = \LL_{X/S}^\vee$ are the relative
cotangent and tangent complexes, with $\LL_X := \LL_{X/\mathrm{pt}}$
and $\TT_X := \TT_{X/\mathrm{pt}}$. An $S$-stack is geometric if its
structure map is $n$-geometric for some $n$, and
$\dSt^{\mathrm{geom}}_S \subseteq \dSt_S$ is the full subcategory of
geometric $S$-stacks \cite[Definition~B.7.1]{CHS2025}.

We write $\Gamma_X$ for global sections of sheaves on $X$, and
$\mathrm{Map}$ both for mapping spaces in an $\infty$-category and
for mapping stacks. References to \cite{CHS2025} use the numbering of
the published memoir, and references to \cite{GinzburgNilpotent} that
of its arXiv version.
\bigskip

\paragraph{\bfem{Acknowledgments.}}
The first author (E.\.{I}.) thanks the second author (K.\.{I}.B.) for
his mentorship, guidance, and support. The question of constructing a
derived Legendrian category was posed by K.\.{I}.B., whose work in
derived contact geometry is the framework of this paper. The
categorical constructions and the proofs of the theorems are due to E.\.{I}.

Both authors thank the Higher Structures Group in the Department of
Mathematics at Middle East Technical University.
\section{Recollections}\label{sec:recollections}

In this section, we recall the results of \cite{PTVV} and
\cite{CHS2025} that we use.

\subsection{Shifted symplectic structures}
Let $Y$ be a derived stack. Denote by $\mathcal{A}^p(Y,s)$ the space
of $s$-shifted $p$-forms on $Y$ and by
$\mathcal{A}^{p,\mathrm{cl}}(Y,s)$ the space of closed $s$-shifted
$p$-forms \cite[Definition~1.12]{PTVV}. They are related by the
derived de Rham differential
\[
  \begin{tikzcd}
    \mathcal{A}^{p}(Y,s) \arrow{r}{d_{dR}} & \mathcal{A}^{p+1,\mathrm{cl}}(Y,s)
  \end{tikzcd}
\]
\cite[Remark~1.17]{PTVV}. Concretely, if $Y = \Spec A$ for a
cofibrant cdga $A$, a closed $s$-shifted $p$-form on $Y$ is
represented by a sequence $\omega = (\omega_0, \omega_1, \dots)$,
where $\omega_i$ has form degree $p+i$ and internal degree $s-i$, and
\[ d\omega_0 = 0, \qquad d_{dR}\omega_i + d\omega_{i+1} = 0 \quad (i \geq 0). \]
The component $\omega_0$ is the underlying $p$-form, and a closed
$2$-form is an $s$-shifted symplectic structure if its underlying
$2$-form induces an equivalence $\TT_Y \to \LL_Y[s]$
\cite[Definition~1.18]{PTVV}. Isotropic and Lagrangian structures on
morphisms are defined in \cite[Definitions~2.7 and~2.8]{PTVV}. For
instance, the derived intersection of two Lagrangians in an
$s$-shifted symplectic stack carries an $(s-1)$-shifted symplectic
structure \cite[Theorem~2.9]{PTVV}.

\subsection{Relative differential forms}\label{subsec:relative-forms}
Let us recall the differential forms of Calaque, Haugseng and
Scheimbauer, which are defined relative to a base derived stack $S$.
For $X \in \dSt_S$, the relative de Rham complex $\DR_S(X)$ is the
sheaf of graded mixed commutative $\cO_S$-algebras on $S$ with values
\[ \DR_S(X)(p \colon \Spec B \to S) = \DR_B(X \times_S \Spec B) \]
\cite[Definition~B.11.4]{CHS2025}, and $\DR(X/S)$ denotes its global
sections, with $\DR(X) := \DR(X/\mathrm{pt})$. Here, for a derived
stack $X$ over $\Spec B$, $\DR_B(X)$ is the limit of $\DR_B$ over the
affine derived schemes over $X$ \cite[Definition~B.11.3]{CHS2025}.
Forms are obtained by mapping into it from a graded mixed module
concentrated in one weight. Let $\mathrm{GMMod}_{\K}$ be the
$\infty$-category of graded mixed $\K$-modules
\cite[Definition~B.11.1]{CHS2025}, and let $\K(p)[-p-s] \in
\mathrm{GMMod}_{\K}$ be the graded mixed $\K$-module given by
$\K[-p-s]$ in weight $p$ and $0$ in the other weights. The spaces of
closed relative $p$-forms and of relative $p$-forms are
\begin{align*}
  \mathcal{A}^{p,\mathrm{cl}}_S(X,s) &:=
  \mathrm{Map}_{\mathrm{GMMod}_{\K}}\big(\K(p)[-p-s], \DR(X/S)\big), \\
  \mathcal{A}^{p}_S(X,s) &:= \mathrm{Map}\big(\K(p)[-p-s], U\DR(X/S)\big),
\end{align*}
where $U$ forgets the mixed structure
\cite[Definition~B.12.1]{CHS2025}. For $S = \mathrm{pt}$ these are
the spaces of forms of \cite[Definition~1.12]{PTVV}. The
constructions of \cite{CHS2025} use the following properties of
relative forms, and in Section~\ref{sec:twisted} we prove their twisted forms.

\begin{proposition}[{\cite{CHS2025}}]\label{prop:CHS-forms}
  Let $S$ be a derived stack.
  \begin{enumerate}
    \item A commutative square of derived stacks
      \[
        \begin{tikzcd}
          X' \arrow{r}{\xi} \arrow{d} & X \arrow{d} \\
          S' \arrow{r}{\sigma} & S
        \end{tikzcd}
      \]
      induces a natural map $\sigma^*\DR_S(X) \to \DR_{S'}(X')$,
      which is an equivalence if the square is cartesian. In the
      cartesian case, if $X$ has a relative cotangent complex over
      $S$, then $\xi^*\LL_{X/S} \simeq \LL_{X'/S'}$.
    \item If $f \colon X \to S$ is geometric, then the underlying
      graded algebra of $\DR(X/S)$ and the spaces of relative forms are
      \begin{align*}
        U\DR(X/S) &\simeq
        \Gamma_X\big(\mathrm{Sym}_{\cO_X}(\LL_{X/S}[-1])\big), \\
        \mathcal{A}^{p}_S(X,s) &\simeq
        \mathrm{Map}_{\QCoh(X)}(\cO_X[-s], \Lambda^p\LL_{X/S}).
      \end{align*}
      If $f$ is a relative Artin stack, then $\LL_{X/S}$ is dualizable, and
      \[ \mathcal{A}^{p}_S(X,s) \simeq
      \mathrm{Map}_{\QCoh(X)}(\Lambda^p\TT_{X/S}, \cO_X[s]). \]
    \item The functors $\mathcal{A}^{p,\mathrm{cl}}_S(-,s)$ and
      $\mathcal{A}^{p}_S(-,s)$ are sheaves for the derived \'etale
      topology and send colimits in $\dSt_S$ to limits. They are
      represented by objects $\mathcal{A}^{p,\mathrm{cl}}_S(s)$ and
      $\mathcal{A}^{p}_S(s)$ of $\dSt_S$. The maps of (1) make closed
      forms functorial in commutative squares, and for $\sigma \colon
      S' \to S$ they give canonical maps
      $\sigma^*\mathcal{A}^{p,\mathrm{cl}}_S(s) \to
      \mathcal{A}^{p,\mathrm{cl}}_{S'}(s)$.
    \item For a pullback square
      \[
        \begin{tikzcd}
          Y \arrow{r} \arrow{d} & X \arrow{d}{f} \\
          T \arrow{r}{\phi} & S
        \end{tikzcd}
      \]
      with $f$ geometric and $\phi$ universally cocontinuous, the natural map
      \[ \Gamma_S\big(\phi_*\cO_T \otimes_S \DR_S(X)\big)
      \longrightarrow \DR(Y/T) \]
      is an equivalence of graded mixed complexes.
  \end{enumerate}
\end{proposition}

Parts (1)--(4) are \cite[Lemmas~B.10.9 and~B.11.5]{CHS2025},
\cite[Theorem~B.10.8, Corollary~B.11.11 and
Proposition~B.12.2]{CHS2025}, \cite[Proposition~B.12.3 and
Remark~B.12.6]{CHS2025}, and \cite[Corollary~B.11.15]{CHS2025}. For a
relative Artin stack $X \to S$, part (2) identifies a $2$-form
$\omega$ with a map
\[
  \begin{tikzcd}
    \TT_{X/S} \arrow{r}{\tilde\omega} & \LL_{X/S}[s],
  \end{tikzcd}
\]
the \bfem{underlying map} of $\omega$. This identification is natural
in $X$ (Proposition~\ref{prop:forms-functor}).

\subsection{Presymplectic stacks and Lagrangian correspondences}
An \bfem{$s$-presymplectic $S$-stack} is an Artin $S$-stack $X$ with
a closed relative $2$-form $\omega \in
\mathcal{A}^{2,\mathrm{cl}}_S(X,s)$. It is \bfem{$s$-symplectic} if
its underlying map $\tilde\omega$ is an equivalence
\cite[Definitions~2.1.2 and~2.1.3]{CHS2025}. By
Proposition~\ref{prop:CHS-forms}(3), an $s$-presymplectic $S$-stack
is the same as a morphism $X \to \mathcal{A}^{2,\mathrm{cl}}_S(s)$ in
$\dSt_S$ with $X$ Artin over $S$. This allows one to define the
$\infty$-category of $s$-presymplectic $S$-stacks as the pullback
\cite[Definition~2.3.2]{CHS2025}
\[ \PreSymp_{S,s} := \dSt^{\mathrm{Art}}_S \times_{\dSt_S}
(\dSt_S)_{/\mathcal{A}^{2,\mathrm{cl}}_S(s)}. \]
It has pullbacks, computed in $\dSt_S$ \cite[Lemma~2.3.7]{CHS2025}.

An \bfem{isotropic structure} on a morphism $f \colon L \to X$ of
Artin $S$-stacks, with $X$ presymplectic, is a path from $0$ to
$f^*\omega$ in $\mathcal{A}^{2,\mathrm{cl}}_S(L,s)$
\cite[Definition~2.4.1]{CHS2025}. It gives a commutative square
\begin{equation}\label{eq:lagrangian-square}
  \begin{tikzcd}
    \TT_{L/S} \arrow{r} \arrow{d} & f^*\TT_{X/S} \arrow{d} \\
    0 \arrow{r} & \LL_{L/S}[s].
  \end{tikzcd}
\end{equation}
Its right vertical map is induced by $\tilde\omega$, and the
isotropic structure is \bfem{Lagrangian} if this square is cartesian
\cite[Definition~2.4.5]{CHS2025}. More generally, a span $X
\xleftarrow{f} L \xrightarrow{g} Y$ in $\PreSymp_{S,s}$ is an
isotropic correspondence, and it gives a commutative square
\begin{equation}\label{eq:lagrangian-correspondence}
  \begin{tikzcd}
    \TT_{L/S} \arrow{r} \arrow{d} & g^*\TT_{Y/S} \arrow{d} \\
    f^*\TT_{X/S} \arrow{r} & \LL_{L/S}[s].
  \end{tikzcd}
\end{equation}
It is a \bfem{Lagrangian correspondence} if this square is cartesian
\cite[Definition~2.4.7]{CHS2025}. These conditions are local on the
base. By Proposition~\ref{prop:CHS-forms}(1), the underlying map of a
presymplectic $S$-stack pulls back to the underlying map of its base
change along $S' \to S$, and equivalences and finite limits in
$\QCoh$ are detected after pullback to affine points
(Proposition~\ref{prop:dag-facts}(1)). The symplectic and Lagrangian
conditions can therefore be checked after base change along all
points $\Spec A \to S$ \cite[Lemmas~2.3.6 and~2.4.8]{CHS2025}.

\subsection{Iterated Lagrangian correspondences}\label{subsec:recollection-lag}
Since $\PreSymp_{S,s}$ has pullbacks, iterated spans in it form an
$(\infty,n)$-category $\Span_n(\PreSymp_{S,s})$
(Theorem~\ref{thm:spans}(1)). Its $i$-morphisms are $i$-fold spans,
which are diagrams in $\PreSymp_{S,s}$ indexed by the partially
ordered set $\mathsf{L}^i$ of Appendix~\ref{app:chs}, and their value
at its initial object is the apex. For $i = 1$ and $i = 2$ these are
the diagrams
\[
  \begin{tikzcd}[column sep=small]
    & L \arrow{dl} \arrow{dr} & \\
    X & & Y,
  \end{tikzcd}
  \qquad\qquad
  \begin{tikzcd}[column sep=small]
    & N \arrow{dl} \arrow{dr} & \\
    L \arrow{d} \arrow{drr} & & L' \arrow{d} \arrow{dll} \\
    X & & Y
  \end{tikzcd}
\]
of spans and of spans between spans. The Lagrangian condition extends
to iterated spans as follows. Namely, an $i$-fold span with apex $N$
determines a diagram in $\QCoh(N)$, whose values are the pullbacks to
$N$ of the tangent complexes and of the shifted cotangent complexes
of all vertices. Here the maps of this diagram are induced by the
maps of the span and by the underlying maps of the $2$-forms. The
span is then called an \bfem{$i$-fold Lagrangian correspondence} if
this diagram is a limit diagram \cite[Definition~2.6.11]{CHS2025}.
For $i = 1$ this is the condition that the square
\eqref{eq:lagrangian-correspondence} is cartesian. We recall this
construction in more detail in Appendix~\ref{app:chs}, since we need
its twisted form.

The following theorem collects the results of \cite{CHS2025} that we build on.

\begin{theorem}[{\cite{CHS2025}}]\label{thm:CHS}
  Let $S$ be a derived stack, $s \in \mathbb{Z}$ and $n \geq 1$.
  \begin{enumerate}
    \item The iterated spans in $\PreSymp_{S,s}$ satisfying the
      Lagrangian condition form a symmetric monoidal
      sub-$(\infty,n)$-category $\Lag_n^{S,s}$ of
      $\Span_n(\PreSymp_{S,s})$, with tensor product and unit
      \[ (X, \omega) \otimes (X', \omega') = (X \times_S X',
      \pi_X^*\omega + \pi_{X'}^*\omega'), \qquad \mathbb{1} = (S, 0). \]
    \item There is an equivalence
      \[ \Lag_n^{S,s}(S,S) \simeq \Lag_{n-1}^{S,s-1}. \]
    \item $\Lag_n^{S,s}$ has adjoints for $1$-morphisms, and the
      adjoint of a Lagrangian correspondence is the reversed
      correspondence. All objects of $\Lag_n^{S,s}$ are fully dualizable.
    \item For an $s$-symplectic $S$-stack $X$, let $\Lag_n^{X/S,s} :=
      \Lag_{n+1}^{S,s}(S,X)$ be the $(\infty,n)$-category of
      Lagrangians in $X$. For objects $Y$ and $Z$ there is an equivalence
      \[ \Lag_n^{X/S,s}(Y,Z) \simeq \Lag_{n-1}^{Z \times_X
      Y^{\mathrm{rev}}/S,\, s-1}, \]
      where $Y^{\mathrm{rev}} \colon X \to S$ is the reversed
      correspondence of (3), and $Z \times_X Y^{\mathrm{rev}}$
      carries the $(s-1)$-symplectic structure of the composite $S
      \to X \to S$ under (2).
    \item For an $s$-symplectic $S$-stack $X$, the AKSZ construction
      of Section~\ref{subsec:recollection-aksz} gives a symmetric
      monoidal functor
      \[
        \begin{tikzcd}
          \Or_n^{\mathrm{cpt}, X\text{-good}, S, d}
          \arrow{r}{\mathrm{AKSZ}} &[1.5em] \Lag_n^{S,s-d},
        \end{tikzcd}
        \qquad T \longmapsto X^T_S.
      \]
      For $d = 0$, its composite with the functors $\mathrm{orcut}$
      and $(-)_B \times S$ is an oriented extended topological field theory
      \[
        \begin{tikzcd}[column sep=3.2em]
          \mathrm{Bord}^{\mathrm{or}}_{0,n} \arrow{r}{\mathrm{orcut}}
          \arrow[bend right=10]{rrr}[swap]{\mathcal{Z}_X} &
          \Or_n^{\mathcal{S},0} \arrow{r}{(-)_B \times S} &[1em]
          \Or_n^{\mathrm{cpt}, X\text{-good}, S, 0}
          \arrow{r}{\mathrm{AKSZ}} & \Lag_n^{S,s}.
        \end{tikzcd}
      \]
  \end{enumerate}
\end{theorem}

Parts (1)--(5) are \cite[Theorem~2.9.5, Definition~2.9.7,
Proposition~2.9.8, Corollaries~2.10.6 and~2.10.8]{CHS2025},
\cite[Proposition~2.10.7]{CHS2025}, \cite[Proposition~2.11.1,
Corollary~2.11.7]{CHS2025}, \cite[Definition~2.11.3,
Corollary~2.11.4]{CHS2025}, and \cite[Corollaries~2.12.19, 3.4.3,
5.6.2 and~5.6.3]{CHS2025}. The homotopy bicategory of
$\Lag_2^{\mathrm{pt},s}$ is the bicategory of
\cite[Corollary~4.12]{AmorimBassat}, as noted in
\cite[\S1.3]{CHS2025}. In (5), $X^T_S$ is the relative mapping stack,
i.e., the internal Hom from $T$ to $X$ in $\dSt_S$
\cite[\S3.2]{CHS2025}, and the $(\infty,n)$-category
$\Or_n^{\mathrm{cpt},X\text{-good},S,d}$ is recalled in
Section~\ref{subsec:recollection-aksz}. Moreover,
$\mathrm{Bord}^{\mathrm{or}}_{0,n}$ is the oriented bordism
$(\infty,n)$-category \cite[Variant~4.1.9 and
Definition~4.1.11]{CHS2025}, oriented extended topological field
theories are as in \cite[Definition~4.1.14]{CHS2025},
$\Or_n^{\mathcal{S},0}$ is the $(\infty,n)$-category of oriented
iterated cospans of finite spaces \cite[Definition~2.12.18]{CHS2025},
$\mathrm{orcut}$ is the functor of \cite[Proposition~5.4.10 and
Corollary~5.6.2]{CHS2025}, and $(-)_B$ is the Betti stack functor
\cite[Definition~2.12.9]{CHS2025}.

The proofs of (1)--(4) in \cite{CHS2025} use differential forms in
four places. First, $\PreSymp_{S,s}$ is defined through the
representing object of Proposition~\ref{prop:CHS-forms}(3), which
gives it pullbacks, while its functoriality in $S$ and the base
change of the Lagrangian conditions come from
Proposition~\ref{prop:CHS-forms}(1). Second, a diagram of
presymplectic stacks is turned into a diagram of quasi-coherent
sheaves, as described above. Third, the symmetric monoidal structure
and (2) come from the pullback square
\begin{equation}\label{eq:CHS-delooping}
  \begin{tikzcd}
    \mathcal{A}^{2,\mathrm{cl}}_S(s-1) \arrow{r} \arrow{d} & S \arrow{d}{0} \\
    S \arrow{r}{0} & \mathcal{A}^{2,\mathrm{cl}}_S(s),
  \end{tikzcd}
\end{equation}
which is natural in $S$ \cite[Proposition~2.10.1]{CHS2025}. By
\cite[Corollary~2.10.3]{CHS2025}, it gives an equivalence
\[ (\PreSymp_{S,s})_{/S,S} \simeq \PreSymp_{S,s-1}. \]
The reason is that a presymplectic $S$-stack with two maps to $S$ is
an Artin $S$-stack with two null-homotopies of its form, and these
form a loop in $\mathcal{A}^{2,\mathrm{cl}}_S(-,s)$. The group
structure on $\mathcal{A}^{2,\mathrm{cl}}_S(s)$ coming from this loop
space description is the addition of forms
\cite[Remark~2.10.2]{CHS2025}. Fourth, $S$ has a unique
$s$-presymplectic structure, since
$\mathcal{A}^{2,\mathrm{cl}}_S(S,s) \simeq *$, and it is symplectic
\cite[Examples~2.1.4(i)]{CHS2025}. In brief, all other steps are
statements about spans in $\infty$-categories with pullbacks, or
about limits of diagrams in stable $\infty$-categories. In
Section~\ref{sec:Leg} we prove the twisted forms of these four inputs
and of the arguments that use them.

\subsection{Orientations and the AKSZ
construction}\label{subsec:recollection-aksz}
A \bfem{$d$-preorientation} of an $S$-stack $T$ is a map $[T] \colon
\Gamma_S\cO_T \to \cO_S[-d]$ in $\QCoh(S)$, and a
\bfem{$d$-preoriented $S$-stack} is a universally cocontinuous
morphism $T \to S$ with a $d$-preorientation
\cite[Definition~2.2.2]{CHS2025}. A preorientation gives a pairing on
sections. For dualizable $\mathcal{E} \in \QCoh(T)$, the pairing
$\mathcal{E} \otimes \mathcal{E}^\vee \to \cO_T$ and $[T]$ induce a map
\[
  \begin{tikzcd}
    \Gamma_S(\mathcal{E}^\vee) \arrow{r} & \Gamma_S(\mathcal{E})^\vee[-d],
  \end{tikzcd}
\]
and $T$ is \bfem{$d$-oriented} if this map is an equivalence for
every dualizable $\mathcal{E}$ after base change along every point
$\Spec A \to S$ \cite[Definitions~2.2.3 and~2.2.4,
Proposition~2.2.5]{CHS2025}. We also need a finiteness condition. A
morphism $T \to S$ is \bfem{$\cO$-compact} if it is universally
cocontinuous and universally perfect, i.e., all its base changes
preserve dualizable objects under pushforward
\cite[Definition~B.10.16]{CHS2025}, and this condition can be checked
after base change to points \cite[Lemma~B.10.18]{CHS2025}. Moreover,
the $d$-preoriented $S$-stacks form an $\infty$-category $\POr_{S,d}$
with pushouts, so iterated cospans in it form an
$(\infty,n)$-category $\Cospan_n(\POr_{S,d})$ \cite[Definition~2.3.8,
Proposition~2.3.11]{CHS2025}. Orientations of cospans are then
defined through diagrams of sections. Let $\Phi$ be an $n$-uple
cospan in $\POr_{S,d}$ with bottom vertex $T$, and let $\mathcal{E}
\in \QCoh(T)$. Then $\Phi$ gives a diagram
$\widehat\Phi_{\mathcal{E}} \colon \widetilde{\mathsf{Sp}}{}^n \to
\QCoh(S)$ of sections of pullbacks of $\mathcal{E}$ and of their
duals, with maps induced by the preorientations. We say that the
cospan is \bfem{oriented} if $\widehat{(\sigma^*\Phi)}_{\mathcal{E}}$
is non-degenerate for every $\sigma \colon S' \to S$ and every
dualizable $\mathcal{E}$ \cite[Corollary~2.7.3, Definitions~2.7.4
and~2.7.5]{CHS2025}, where $\widetilde{\mathsf{Sp}}{}^n$ and
non-degeneracy are as in Appendix~\ref{app:chs}.

Let $T$ be a $d$-preoriented $S$-stack and $X$ a geometric $S$-stack,
and let $Y := X \times_S T$. Proposition~\ref{prop:CHS-forms}(4) and
$[T]$ give maps of graded mixed complexes
\[
  \begin{tikzcd}
    \DR(Y/T) \simeq \Gamma_S\big(\phi_*\cO_T \otimes_S \DR_S(X)\big)
    \arrow{r}{[T]} & \Gamma_S\big(\cO_S[-d] \otimes_S \DR_S(X)\big)
    \simeq \DR(X/S)[-d],
  \end{tikzcd}
\]
where $\phi \colon T \to S$ and $f \colon X \to S$ are the structure
maps. Mapping out of $\K(p)[-p-s]$ gives the \bfem{pushforward}
\[
  \begin{tikzcd}
    \mathcal{A}^{p,\mathrm{cl}}_T(Y,s) \arrow{r}{\int_{[T]}} &
    \mathcal{A}^{p,\mathrm{cl}}_S(X,s-d)
  \end{tikzcd}
\]
of \cite[Proposition~3.1.1]{CHS2025}, which is compatible with base
change \cite[Remark~3.1.4]{CHS2025}. When $X$ is an Artin $S$-stack,
the pushforward can be described on underlying forms. Suppose that
the underlying form of $\omega$ corresponds to $\omega'' \colon
\Lambda^p\TT_{X/S} \to f^*\phi_*\cO_T[s]$. Then the underlying form
of $\int_{[T]}\omega$ corresponds to the composite
\begin{equation}\label{eq:CHS-pushforward}
  \begin{tikzcd}
    \Lambda^p\TT_{X/S} \arrow{r}{\omega''} & f^*\phi_*\cO_T[s]
    \arrow{r}{f^*[T]} & \cO_X[s-d]
  \end{tikzcd}
\end{equation}
\cite[Remark~3.1.2]{CHS2025}. The AKSZ construction combines this
pushforward with pullback along an evaluation map. Let $X$ be an
Artin $S$-stack and $T$ an $\cO$-compact $d$-preoriented $S$-stack
such that the relative mapping stack $X^T_S$ is Artin over $S$. The
evaluation and projection maps
\[
  \begin{tikzcd}
    X^T_S \times_S T \arrow{r}{\mathrm{ev}} \arrow{d}[swap]{\pi} & X \\
    X^T_S &
  \end{tikzcd}
\]
satisfy $\TT_{X^T_S/S} \simeq \pi_*\mathrm{ev}^*\TT_{X/S}$
\cite[Proposition~B.10.21]{CHS2025}. The \bfem{AKSZ construction}
sends $\omega \in \mathcal{A}^{2,\mathrm{cl}}_S(X,s)$ to
$\int_{[T]}\mathrm{ev}^*\omega \in
\mathcal{A}^{2,\mathrm{cl}}_S(X^T_S, s-d)$, where
$\mathrm{ev}^*\omega$ is the pullback of $\omega$ along the square
formed by $\mathrm{ev}$ over $T \to S$. Under the identification of
tangent complexes, its underlying $2$-form is the composite
\begin{equation}\label{eq:CHS-aksz}
  \begin{tikzcd}[column sep=3em]
    \Lambda^2\pi_*\mathrm{ev}^*\TT_{X/S} \arrow{r} &
    \pi_*\Lambda^2\mathrm{ev}^*\TT_{X/S}
    \arrow{r}{\pi_*\mathrm{ev}^*\omega} & \pi_*\mathrm{ev}^*\cO_X[s]
    \arrow{d}{\simeq} \\
    & \cO_{X^T_S}[s-d] & \sigma^*\phi_*\cO_T[s] \arrow{l}[swap]{\sigma^*[T]}
  \end{tikzcd}
\end{equation}
where $\sigma \colon X^T_S \to S$ is the structure map
\cite[Proposition~3.2.1]{CHS2025}. If $X$ is symplectic and $T$ is
$d$-oriented, this form is symplectic, and the same holds for the
diagrams attached to iterated oriented cospans. This is the
non-degeneracy of the AKSZ construction
\cite[Proposition~3.4.2]{CHS2025}, which gives
Theorem~\ref{thm:CHS}(5). In that theorem,
$\POr^{X\text{-good}}_{S,d} \subseteq \POr_{S,d}$ is spanned by the
$T$ such that $X^T_S$ is Artin over $S$. The objects of the
$(\infty,n)$-category $\Or_n^{\mathrm{cpt}, X\text{-good}, S, d}$ are
the $\cO$-compact $d$-oriented such $T$, and its morphisms are the
iterated oriented cospans between them \cite[\S3.3 and
Notation~3.4.1]{CHS2025}. The AKSZ construction is compatible with
the delooping equivalence coming from \eqref{eq:CHS-delooping}
\cite[Proposition~3.3.3]{CHS2025}, and it gives a symmetric monoidal
functor $\Cospan_n(\POr^{X\text{-good}}_{S,d}) \to
\Span_n(\PreSymp_{S,s-d})$ \cite[Corollaries~3.3.2 and~3.3.5]{CHS2025}.

\section{Shifted contact structures}\label{sec:contact}

\subsection{Weights}
Let $X$ be a derived Artin stack and $\cL$ a line bundle on $X$,
classified by $f \colon X \to B\Gm$. Let $z$ be the coordinate on
$\Gm = \mathrm{pt} \times_{B\Gm} \mathrm{pt}$. Grading by powers of $z$ gives
\[ q_*\cO_{\mathrm{pt}} \simeq \bigoplus_{w \in \mathbb{Z}}
\mathcal{M}^{\otimes w}, \]
where $\mathcal{M}$ is the summand of degree one. We normalize the
classifying map so that $f^*\mathcal{M} \simeq \cL$. The principal
$\Gm$-bundle of $\cL$ is the pullback
\[
  \begin{tikzcd}
    \widetilde X \arrow{r} \arrow{d}[swap]{p} & \mathrm{pt} \arrow{d}{q} \\
    X \arrow{r}{f} & B\Gm,
  \end{tikzcd}
  \qquad p_*\cO_{\widetilde X} \simeq f^*q_*\cO_{\mathrm{pt}} \simeq
  \bigoplus_{w \in \mathbb{Z}} \cL^{\otimes w}.
\]
Thus $\widetilde X \simeq \mathbf{Spec}_X \big(\bigoplus_{w}
\cL^{\otimes w}\big)$, and $\Gm$ acts on $\widetilde X$ with
$\cL^{\otimes w}$ as its weight-$w$ part.

\begin{definition}\label{def:weights}
  Since $f$ is geometric and $q$ is affine,
  Proposition~\ref{prop:CHS-forms}(4) gives the first equivalence
  below. Here $q$ is universally cocontinuous in the sense of
  \cite[Definition~B.8.7]{CHS2025}, because affine morphisms are
  cocontinuous \cite[Corollary~B.8.5(i)]{CHS2025} and base changes of
  affine morphisms are affine. We put
  \begin{align*}
    \DR(\widetilde X) &\simeq \Gamma_{B\Gm}\big(q_*\cO_{\mathrm{pt}}
    \otimes \DR_{B\Gm}(X)\big), \\
    \DR(\widetilde X)^{(w)} &:=
    \Gamma_{B\Gm}\big(\mathcal{M}^{\otimes w} \otimes \DR_{B\Gm}(X)\big), \\
    \mathcal{A}^{p,\mathrm{cl}}(\widetilde X,s)^{(w)} &:=
    \mathrm{Map}_{\mathrm{GMMod}_{\K}}\big(\K(p)[-p-s],
    \DR(\widetilde X)^{(w)}\big),
  \end{align*}
  where $\DR(\widetilde X)^{(w)} \to \DR(\widetilde X)$ is induced by
  the inclusion of the summand $\mathcal{M}^{\otimes w}$. We call the
  points of $\mathcal{A}^{p,\mathrm{cl}}(\widetilde X,s)^{(w)}$ the
  \bfem{weight-$w$ closed $s$-shifted $p$-forms} on $\widetilde X$,
  and $\mathcal{A}^{p}(\widetilde X,s)^{(w)}$ is defined in the same
  way using underlying graded modules. The grading of
  $q_*\cO_{\mathrm{pt}}$ by powers of $z$ is the weight grading for
  the $\Gm$-action on $\widetilde X$, so a form $\omega$ in the image
  of $\mathcal{A}^{p,\mathrm{cl}}(\widetilde X,s)^{(w)}$ satisfies
  \[ \rho^*\omega \simeq z^{w} \cdot \pi^*\omega \quad \text{in }
  \mathcal{A}^{p,\mathrm{cl}}_{\Gm}(\Gm \times \widetilde X, s), \]
  where $\rho, \pi \colon \Gm \times \widetilde X \to \widetilde X$
  are the action and projection maps, and both sides are regarded as
  forms relative to the first factor $\Gm$. As absolute forms the two
  sides differ in general (see Example~\ref{ex:weights}). For $w = 1$
  this is the condition of \cite[Definition~3.6]{IzbudakBerktav2026}.
\end{definition}

\begin{lemma}\label{lem:weights}
  Let $X$ and $\cL$ be as in Definition~\ref{def:weights}.
  \begin{enumerate}
    \item For $E \in \QCoh(X)$ and $w \in \mathbb{Z}$, let $\tau
      \colon p^*\cL \simeq \cO_{\widetilde X}$ be the tautological
      trivialization. The map
      \[
        \begin{tikzcd}
          \mathrm{Map}_{\QCoh(X)}(\cO_X[-s], E \otimes \cL^{\otimes
          w}) \arrow{d}[swap]{\sigma \,\mapsto\, \tau^{w}(p^*\sigma)} \\
          \mathrm{Map}_{\QCoh(\widetilde X)}(\cO_{\widetilde X}[-s], p^*E)
        \end{tikzcd}
      \]
      identifies its source with a retract of its target, the
      weight-$w$ part. Moreover,
      \begin{align*}
        p^*\LL_{X/B\Gm} &\simeq \LL_{\widetilde X}, \\
        \mathcal{A}^p(\widetilde X, s)^{(w)} &\simeq
        \mathrm{Map}_{\QCoh(X)}\big(\cO_X[-s], \Lambda^p\LL_{X/B\Gm}
        \otimes \cL^{\otimes w}\big).
      \end{align*}
    \item The derived de Rham differential preserves weights, and so
      does pullback along morphisms of derived stacks over $B\Gm$.
    \item If $X = \Spec A$ and $\cL$ is trivial, then $\widetilde X =
      \Spec A[t^{\pm 1}]$, where $t$ has weight one, and
      \[ \DR(\widetilde X)^{(w)} = t^w \DR(A) \oplus t^w \DR(A)\,
        d_{dR}t/t \ \subset\ \DR(A) \otimes \DR(\K[t^{\pm1}]) \simeq
      \DR(A[t^{\pm1}]). \]
  \end{enumerate}
\end{lemma}

\begin{proof}\mbox{}
  \begin{enumerate}[wide]
    \item By the projection formula and $p_*\cO_{\widetilde X} \simeq
      \bigoplus_w \cL^{\otimes w}$,
      \[ \mathrm{Map}(\cO_{\widetilde X}[-s], p^*E) \simeq
        \mathrm{Map}(\cO_X[-s], E \otimes p_*\cO_{\widetilde X}) \simeq
        \mathrm{Map}\Big(\cO_X[-s], \bigoplus_{w \in \mathbb{Z}} E
      \otimes \cL^{\otimes w}\Big). \]
      The inclusion of and the projection onto the summand $E \otimes
      \cL^{\otimes w}$ exhibit $\mathrm{Map}(\cO_X[-s], E \otimes
      \cL^{\otimes w})$ as a retract, which is the weight-$w$
      summand. Unwinding the projection formula gives the stated map.
      The first equivalence of the second display is
      Proposition~\ref{prop:CHS-forms}(1), and the second follows
      from Lemma~\ref{lem:twisted}(2) below with $S = B\Gm$.

    \item The equivalence of Proposition~\ref{prop:CHS-forms}(4) is
      natural in $X \in \dSt_{B\Gm}$ and compatible with the graded
      mixed structures.

    \item If $\cL$ is trivial, then $f$ factors through $q$, and
      \[ \widetilde X \simeq X \times (\mathrm{pt} \times_{B\Gm}
      \mathrm{pt}) = X \times \Gm \]
      compatibly with the $\Gm$-actions. Under this identification
      the grading of $q_*\cO_{\mathrm{pt}}$ by powers of $z$ pulls
      back to the grading of $\cO(\widetilde X) = A[t^{\pm1}]$ by
      powers of $t = z$. By the naturality in (2), the summand
      $\DR(\widetilde X)^{(w)}$ is therefore the part of
      $\DR(A[t^{\pm1}])$ on which $\Gm$ acts with weight $w$, and
      since $d_{dR}t/t$ has weight zero, this is the stated part. \qedhere
  \end{enumerate}
\end{proof}

\begin{example}\label{ex:weights}
  Let $X = \mathrm{pt}$ and $\cL = \cO_{\mathrm{pt}}$, so that
  $\widetilde X = \Gm = \Spec \K[t^{\pm1}]$, with $\Gm$ acting by
  $\rho^*t = zt$. By Lemma~\ref{lem:weights}(3), the weight-$w$ part
  of $\DR(\K[t^{\pm1}])$ is spanned by $t^w$ and $t^w\, d_{dR}t/t$.
  Thus the function $t^w$ has weight $w$, the $1$-form $d_{dR}t$ has
  weight one, and $d_{dR}t/t$ has weight zero. Relative to the first
  factor of $\Gm \times \Gm$ we have $\rho^*(d_{dR}t) = d_{dR}(zt) =
  z\, d_{dR}t$, as in Definition~\ref{def:weights}. As an absolute
  form, however, $\rho^*(d_{dR}t) = z\, d_{dR}t + t\, d_{dR}z$.
\end{example}

\subsection{Contact structures and Legendrians via symplectification}

\begin{definition}\label{def:contact}
  An \bfem{$s$-shifted contact structure} on a derived Artin stack
  $X$ is a pair $(\cL, \omega)$. Here $\cL$ is a line bundle on $X$,
  and $\omega \in \mathcal{A}^{2,\mathrm{cl}}(\widetilde X, s)^{(1)}$
  is a weight-one closed $2$-form whose image in
  $\mathcal{A}^{2,\mathrm{cl}}(\widetilde X, s)$ is an $s$-shifted
  symplectic structure. The $s$-shifted symplectic stack $(\widetilde
  X, \omega)$ is the \bfem{symplectification} of $(X, \cL, \omega)$.
\end{definition}

\begin{example}\label{ex:contact}
  The following derived stacks carry shifted contact structures.
  \begin{enumerate}
    \item Let $X = \mathbb{A}^{2m+1} = \Spec
      \K[x_1,\dots,x_m,y_1,\dots,y_m,z]$ and $\cL = \cO_X$, so that
      $\widetilde X = X \times \Gm$ with coordinate $t$ on $\Gm$, and
      let $\alpha := d_{dR}z - \sum_i y_i\, d_{dR}x_i$. The $1$-form
      $t\alpha$ on $\widetilde X$ has weight one by
      Lemma~\ref{lem:weights}(3), so its derived de Rham differential
      \[ \omega := d_{dR}(t\alpha) = d_{dR}t \wedge \alpha - t \sum_i
      d_{dR}y_i \wedge d_{dR}x_i \]
      is a weight-one closed $2$-form by Lemma~\ref{lem:weights}(2).
      Since $\widetilde X$ is smooth, the underlying $2$-form of
      $\omega$ is the classical $2$-form $d(t\alpha)$. It is
      non-degenerate, because $\omega^{m+1} = (m+1)\, t^m\, d_{dR}t
      \wedge \alpha \wedge (d_{dR}\alpha)^m$ vanishes nowhere. Thus
      $\omega$ is a $0$-shifted symplectic structure
      \cite[Definition~1.18]{PTVV}, and $(X, \cO_X, \omega)$ is a
      $0$-shifted contact scheme, the standard contact structure on
      affine space.
    \item Let $X = B\Gm$ with the weight-one line bundle, so that
      $\widetilde X = \mathrm{pt}$. By Lemma~\ref{lem:twisted}(6) and
      Proposition~\ref{prop:relative}(1), the space of weight-one
      closed $s$-shifted $2$-forms on $\mathrm{pt}$ is contractible,
      and its point is symplectic since $\TT_{\mathrm{pt}} \simeq 0$,
      so $B\Gm$ is $s$-shifted contact for every $s \in \mathbb{Z}$.
      It is the unit of the contact product (cf.
      Proposition~\ref{prop:relative}(3)).
    \item For a derived Artin stack $Y$, the projectivized shifted
      cotangent stack $[T^*[s]Y/\Gm]$ is $s$-shifted contact, with
      symplectification $T^*[s]Y$ (cf. Example~\ref{ex:PT}).
  \end{enumerate}
\end{example}

\begin{definition}\label{def:legendrian}
  Let $(X,\cL,\omega)$ be an $s$-shifted contact derived Artin stack
  and $g \colon Y \to X$ a morphism of derived Artin stacks. The
  principal $\Gm$-bundle of $g^*\cL$ is the pullback
  \[
    \begin{tikzcd}
      \widetilde Y \arrow{r}{\tilde g} \arrow{d}[swap]{p_Y} &
      \widetilde X \arrow{d}{p} \\
      Y \arrow{r}{g} & X.
    \end{tikzcd}
  \]
  The map $\tilde g$ is $\Gm$-equivariant. A \bfem{Legendrian
  structure} on $g$ is a path
  \[ \gamma \colon 0 \rightsquigarrow \tilde g^*\omega \quad \text{in
  } \mathcal{A}^{2,\mathrm{cl}}(\widetilde Y, s)^{(1)} \]
  whose image in $\mathcal{A}^{2,\mathrm{cl}}(\widetilde Y, s)$ is a
  Lagrangian structure on $\tilde g$ in the sense of
  \cite[Definition~2.8]{PTVV}. A \bfem{Legendrian morphism} into $X$
  is a morphism with a Legendrian structure.
\end{definition}

A \bfem{Legendrian correspondence} from $X$ to $Y$ is a Legendrian
morphism into the contact product $\overline{X} \circledast Y$ of
Proposition~\ref{prop:relative}, where $\overline{X}$ carries the
negated contact structure. Classically, Legendrian submanifolds of a
contact manifold correspond to conic Lagrangian submanifolds of its
symplectification. By Theorem~\ref{thm:legendrian-comparison},
Definition~\ref{def:legendrian} agrees with Berktav's definition
(Definition~\ref{def:legendrian-berktav}) where the latter applies.

\begin{example}\label{ex:legendrian}
  We describe Legendrians in the first two cases of Example~\ref{ex:contact}.
  \begin{enumerate}
    \item In Example~\ref{ex:contact}(1), let $i \colon L =
      \mathbb{A}^m \hookrightarrow \mathbb{A}^{2m+1}$ be the
      subscheme $\{y_1 = \dots = y_m = z = 0\}$, so that $\widetilde
      L = L \times \Gm$. The pullback of $\alpha$ to $L$ vanishes,
      since $z$ and the $y_i$ vanish on $L$, so the pullback of
      $t\alpha$ to $\widetilde L$ is zero. Applying $d_{dR}$ gives a
      weight-one path from $0$ to $\tilde\imath^*\omega$, i.e., a
      weight-one isotropic structure on $\tilde\imath$. As in the
      proof of Lemma~\ref{lem:smooth-legendrian}, it is Lagrangian,
      since $\TT_{\widetilde L}$ is an isotropic subbundle of
      $\tilde\imath^*\TT_{\widetilde X}$ of half the rank, so $i$ is
      a Legendrian morphism. This is a special case of
      Lemma~\ref{lem:smooth-legendrian}, and $L$ is the $1$-jet of
      the zero function in the sense of Section~\ref{sec:applications}.
    \item A Legendrian morphism into the unit $B\Gm$ is a morphism $L
      \to B\Gm$, i.e., a line bundle on $L$, together with a
      weight-one Lagrangian structure on $\widetilde L \to
      \mathrm{pt}$. Such a structure is the same as a weight-one
      $(s-1)$-shifted symplectic structure on $\widetilde L$, i.e.,
      an $(s-1)$-shifted contact structure on $L$ (cf. Step~4 of the
      proof of Theorem~\ref{thm:transfer}).
  \end{enumerate}
\end{example}

\subsection{Berktav's contact structures and Legendrians}

For a line bundle $\cL$ on $X$, denote by
\[ \mathcal{A}^1(X,\cL;s) := \mathrm{Map}_{\QCoh(X)}(\cO_X, \LL_X
\otimes \cL[s]) \]
the space of $\cL$-twisted $s$-shifted $1$-forms. A point $\alpha$ is
equivalently a map $\alpha^\vee \colon \TT_X \to \cL[s]$.

\begin{definition}[{\cite[Definitions~2.16 and~2.17]{kib2},
  \cite[Definition~2.38]{IzbudakBerktav2026}}]\label{def:berktav}
  Let $\alpha \in \mathcal{A}^1(X,\cL;s)$ and
  \[ \mathcal{K} := \mathrm{fib}\big(\alpha^\vee \colon \TT_X \to
  \cL[s]\big). \]
  The form $\alpha$ is an \bfem{$s$-shifted contact form} if, locally
  on $X$ where $\cL$ is trivialized, the map
  \[
    \begin{tikzcd}
      \mathcal{K} \arrow{r}{d_{dR}\alpha|_{\mathcal{K}}} &[1.5em]
      \mathcal{K}^\vee \otimes \cL[s]
    \end{tikzcd}
  \]
  induced by the underlying $2$-form of $d_{dR}\alpha$ is an equivalence.
\end{definition}

\begin{example}\label{ex:berktav-contact}
  The form $\alpha = d_{dR}z - \sum_i y_i\, d_{dR}x_i$ on
  $\mathbb{A}^{2m+1}$, with $\cL$ trivial, is a $0$-shifted contact
  form in the sense of Definition~\ref{def:berktav}. In fact,
  $\mathcal{K}$ is the subbundle of $\TT_X$ spanned by the
  $\partial_{x_i} + y_i\partial_z$ and the $\partial_{y_i}$, and
  $d_{dR}\alpha = -\sum_i d_{dR}y_i \wedge d_{dR}x_i$ pairs
  $\partial_{x_i} + y_i\partial_z$ with $\partial_{y_j}$ by
  $\delta_{ij}$, while it pairs any two vectors of the same kind to
  zero. Its restriction to $\mathcal{K}$ is therefore non-degenerate.
  Under the equivalence $\Phi$ of Lemma~\ref{lem:exact} below,
  $\alpha$ corresponds to the weight-one form $\omega =
  d_{dR}(t\alpha)$ of Example~\ref{ex:contact}(1). More generally,
  \cite[Theorem~4.2]{kib2} equips shifted $1$-jet stacks with shifted
  contact structures in this sense.
\end{example}

\begin{remark}\label{rem:locally}
  If the trivialization of $\cL$ is changed by a unit $u$, then
  \[ d_{dR}(u\alpha)|_{\mathcal{K}} = \big(d_{dR}u \wedge \alpha +
  u\, d_{dR}\alpha\big)|_{\mathcal{K}} = u\, d_{dR}\alpha|_{\mathcal{K}}, \]
  since $\alpha$ vanishes on $\mathcal{K}$. Hence the condition in
  Definition~\ref{def:berktav} does not depend on the trivialization.
  In that respect, we need to specify which covers are allowed. In
  his work on derived Artin stacks, Berktav works over open substacks
  on which $\cL$ is trivial \cite[proofs of Theorems~3.7 and~3.9, and
  footnote~10]{kib2}. Here we read ``locally'' as referring to an
  \'etale cover of $X$ by derived stacks on which $\cL$ is trivial,
  because tangent complexes and forms pull back along \'etale maps.
  Along a smooth cover $U \to X$ they do not, since $\TT_U$ is larger
  than the pullback of $\TT_X$. For derived schemes, on the other
  hand, a Zariski cover suffices, as in \cite[Definition~4.3,
  Step~1]{kib1}. Note that an \'etale cover trivializing $\cL$ need
  not exist. On $B\Gm$ with the weight-one line bundle, for example,
  the stabilizer $\Gm$ acts on the fiber with weight one, and the
  same holds along the zero section of the projectivized cotangent
  stacks of Example~\ref{ex:PT}. This condition is also needed in
  Definition~\ref{def:legendrian-berktav} below, while
  Definitions~\ref{def:contact} and~\ref{def:legendrian} do not need it.
\end{remark}

To compare with Berktav's definitions we use the Euler vector field.
The \bfem{Euler vector field} $E \colon \cO_{\widetilde X} \to
\TT_{\widetilde X}$ is the fundamental vector field of the
$\Gm$-action $\rho \colon \Gm \times \widetilde X \to \widetilde X$.
More precisely, restricting $\rho$ along $\Spec
\K[\epsilon]/(\epsilon^2) \times \widetilde X \to \Gm \times
\widetilde X$, where the first factor is given by $z \mapsto 1 +
\epsilon$, gives a map $\Spec \K[\epsilon]/(\epsilon^2) \times
\widetilde X \to \widetilde X$ whose restriction to $\epsilon = 0$ is
the identity, and by the universal property of the cotangent complex
this is a derivation $\iota_E \colon \LL_{\widetilde X} \to
\cO_{\widetilde X}$. Its dual is $E$. For example, if $X = \Spec A$
and $\cL$ is trivial, then $\rho^*t = zt$ on $\widetilde X = \Spec
A[t^{\pm1}]$ by Lemma~\ref{lem:weights}(3), and hence $E =
t\partial_t$. This map $E$ is also the fundamental morphism of
\cite[Corollary~3.5]{IzbudakBerktav2026}. In
\cite[Definition~3.6]{BruceGrabowskaGrabowski2017}, a contact
structure is a principal $\mathbb{R}^\times$-bundle with a symplectic
form of homogeneity one. The fundamental vector field of the
$\mathbb{R}^\times$-action is called the Euler vector field there
\cite[\S2.1, after Proposition~2.1]{BruceGrabowskaGrabowski2017}, and
contracting the symplectic form with it gives $t\alpha$ for the
contact form $\alpha$ \cite[\S3.3, after
Proposition~3.4]{BruceGrabowskaGrabowski2017}. Lemma~\ref{lem:exact}
below is the derived analogue of this statement. In the derived
setting, we use the fiber sequences
\begin{equation*}\label{eq:atiyah}\tag{A}
  \begin{tikzcd}
    p^*\LL_X \arrow{r} & \LL_{\widetilde X} \arrow{r}{\iota_E} &
    \cO_{\widetilde X}, \\
    \cO_{\widetilde X} \arrow{r}{E} & \TT_{\widetilde X} \arrow{r} & p^*\TT_X
  \end{tikzcd}
\end{equation*}
of \cite[Lemma~3.4, Corollary~3.5]{IzbudakBerktav2026}. The second is
dual to the first. The space of horizontal weight-one $1$-forms is
\[ \mathcal{A}^1(\widetilde X,s)^{(1),\mathrm{hor}} :=
  \mathrm{fib}\big(\iota_E \colon \mathcal{A}^1(\widetilde X,s)^{(1)}
\to \mathcal{A}^0(\widetilde X,s)^{(1)}\big). \]
By \eqref{eq:atiyah} and Lemma~\ref{lem:weights}(1), there is an equivalence
\[
  \begin{tikzcd}
    \mathcal{A}^1(X,\cL;s) \arrow{r}{\sim} & \mathcal{A}^1(\widetilde
    X,s)^{(1),\mathrm{hor}},
  \end{tikzcd}
  \qquad \alpha \longmapsto p^\sharp\alpha.
\]

\begin{lemma}\label{lem:exact}
  The composite $\Phi$ in the diagram
  \[
    \begin{tikzcd}[column sep=2.4em]
      \mathcal{A}^1(X,\cL;s) \arrow{r}{\sim} \arrow[bend
      right=10]{rrr}[swap]{\Phi} & \mathcal{A}^1(\widetilde
      X,s)^{(1),\mathrm{hor}} \arrow{r} & \mathcal{A}^1(\widetilde
      X,s)^{(1)} \arrow{r}{d_{dR}} &
      \mathcal{A}^{2,\mathrm{cl}}(\widetilde X,s)^{(1)}
    \end{tikzcd}
  \]
  is an equivalence of spaces, natural in $(X,\cL)$ with respect to
  morphisms over $B\Gm$. Its inverse sends $\omega$ to the form
  corresponding to the contraction $\iota_E\omega_0$ of its underlying $2$-form.
\end{lemma}

\begin{proof}
  Let us first reduce to an affine $X$ with trivial $\cL$, and then
  compute with a contracting homotopy.

  \bfem{Step 1.} By \eqref{eq:atiyah}, the source of $\Phi$ is the
  fiber of $\iota_E \colon \mathcal{A}^1(\widetilde X,s)^{(1)} \to
  \mathcal{A}^0(\widetilde X,s)^{(1)}$. By
  Definition~\ref{def:weights}, these spaces and the target of $\Phi$
  are spaces of twisted relative forms,
  \[ \mathcal{A}^{p}(\widetilde X,s)^{(1)} =
    \mathcal{A}^{p}_{B\Gm}(X,s;\mathcal{M}), \qquad
    \mathcal{A}^{2,\mathrm{cl}}(\widetilde X,s)^{(1)} =
  \mathcal{A}^{2,\mathrm{cl}}_{B\Gm}(X,s;\mathcal{M}), \]
  which send colimits in $\dSt_{B\Gm}$ to limits by
  Lemma~\ref{lem:twisted}(3). Since line bundles on affine derived
  schemes are trivial on a cover by basic open affines, every derived
  stack over $B\Gm$ is a colimit of affine derived schemes on which
  $\cL$ is trivial. Since $\Phi$ is natural, we may assume that $X =
  \Spec A$ for a cofibrant cdga $A$ and that $\cL$ is trivial.

  \medskip
  \bfem{Step 2.} By Lemma~\ref{lem:weights}(3), $\widetilde X = \Spec
  A[t^{\pm1}]$ and $E = t\partial_t$. Let $W_k$ be the weight-one
  part of the forms of form degree $k$ on $\widetilde X$, and let $H
  := \iota_E$. Then
  \begin{align*}
    W_k &= t\,\Omega^k_A \oplus t\,\Omega^{k-1}_A\, d_{dR}t/t, \\
    dH + Hd &= 0, \\
    d_{dR}H + H d_{dR} &= L_E = \mathrm{id} \quad \text{on } W_k.
  \end{align*}
  The second line holds because $E$ commutes with the internal
  differential $d$. The third is the Cartan formula, since the Lie
  derivative $L_E$ acts on weight-one forms as the identity. Now let
  $T^{\geq k} := \prod_{i \geq 0} W_{k+i}$, with total differential
  $D = d + d_{dR}$, so that closed weight-one $2$-forms are cocycles
  in $T^{\geq 2}$. The last two lines give $DH + HD = \mathrm{id}$ on
  $T^{\geq 0}$, so $T^{\geq 0}$ is contractible. With $T^{<2} := (W_0
  \oplus W_1, D)$, the cofiber sequence $T^{\geq 2} \to T^{\geq 0}
  \to T^{<2}$ therefore gives an equivalence
  \[
    \begin{tikzcd}
      T^{<2}[-1] \arrow{r}{\sim} & T^{\geq 2},
    \end{tikzcd}
    \qquad (h, \lambda) \longmapsto (d_{dR}\lambda, 0, 0, \dots).
  \]
  On $W_0$ we have $H d_{dR} = \mathrm{id}$, so $d_{dR} \colon W_0
  \to W_1$ is split injective, compatibly with $d$. Its image is a
  complement of $\ker(H|_{W_1})$. Therefore
  \[ T^{<2} \simeq \big(\ker(H|_{W_1}), d\big) = \big(t\,\Omega^1_A, d\big), \]
  which computes the weight-one horizontal $1$-forms, i.e.,
  $\mathcal{A}^1(X,\cL;s)$. The composite equivalence is $\lambda
  \mapsto (d_{dR}\lambda,0,0,\dots)$, which is $\Phi$. Its inverse is
  given on underlying forms by $H$, i.e., by $\omega \mapsto \iota_E\omega_0$.
\end{proof}

To compare the two notions of non-degeneracy, we use the following
reduction of a $2$-form along a map $\cO_Y \to V$.

\begin{lemma}\label{lem:reduction}
  Let $Y$ be a derived stack, $V \in \QCoh(Y)$ perfect, $e \colon
  \cO_Y \to V$ a map, and $\omega$ an $s$-shifted $2$-form on $V$
  with adjoint $\omega^\flat \colon V \to V^\vee[s]$. Let $\lambda :=
  \omega^\flat \circ e$ be the composite
  \[
    \begin{tikzcd}
      \cO_Y \arrow{r}{e} \arrow[bend right=20]{rr}[swap]{\lambda} & V
      \arrow{r}{\omega^\flat} & V^\vee[s],
    \end{tikzcd}
  \]
  and $\lambda^\vee \colon V \to \cO_Y[s]$ its dual. Since
  $\Lambda^2\cO_Y \simeq 0$, there is a canonical null-homotopy of
  $\omega(e,e)$, and it gives factorizations $e' \colon \cO_Y \to
  e^{\perp} := \mathrm{fib}(\lambda^\vee)$ of $e$ and $\lambda'
  \colon \cO_Y \to (V/e)^\vee[s]$ of $\lambda$. Let $\mathcal{K} :=
  \mathrm{cofib}(e')$, so that
  \[ \mathcal{K} \simeq \mathrm{fib}\big(\bar\alpha \colon V/e \to
  \cO_Y[s]\big), \]
  with $\bar\alpha$ induced by $\lambda^\vee$. Then $\omega^\flat$
  restricts to a map $e^\perp \to (V/e)^\vee[s]$ compatible with $e'$
  and $\lambda'$, and it induces $\bar\omega \colon \mathcal{K} \to
  \mathcal{K}^\vee[s]$. The map $\omega^\flat$ is an equivalence if
  and only if $\bar\omega$ is.
\end{lemma}

\begin{proof}
  By the antisymmetry of $\omega$, the composite $e^\vee[s] \circ
  \omega^\flat \colon V \to \cO_Y[s]$ is identified with
  $\pm\lambda^\vee$, so $\omega^\flat$ carries $e^\perp$ into
  $\mathrm{fib}(e^\vee[s]) = (V/e)^\vee[s]$. We obtain a map of fiber sequences
  \[
    \begin{tikzcd}
      e^\perp \arrow{r} \arrow{d} & V \arrow{r}{\lambda^\vee}
      \arrow{d}{\omega^\flat} & \cO_Y[s] \arrow{d}{\pm 1} \\
      (V/e)^\vee[s] \arrow{r} & V^\vee[s] \arrow{r}{e^\vee[s]} & \cO_Y[s],
    \end{tikzcd}
  \]
  so $\omega^\flat$ is an equivalence if and only if $e^\perp \to
  (V/e)^\vee[s]$ is. The dual of the fiber sequence $\mathcal{K} \to
  V/e \xrightarrow{\bar\alpha} \cO_Y[s]$, shifted by $s$, is $\cO_Y
  \xrightarrow{\lambda'} (V/e)^\vee[s] \to \mathcal{K}^\vee[s]$.
  Since both factorizations come from the same null-homotopy of
  $\omega(e,e)$, we get a map of cofiber sequences
  \[
    \begin{tikzcd}
      \cO_Y \arrow{r}{e'} \arrow[equal]{d} & e^\perp \arrow{r}
      \arrow{d} & \mathcal{K} \arrow{d}{\bar\omega} \\
      \cO_Y \arrow{r}{\lambda'} & (V/e)^\vee[s] \arrow{r} & \mathcal{K}^\vee[s],
    \end{tikzcd}
  \]
  so $e^\perp \to (V/e)^\vee[s]$ is an equivalence if and only if
  $\bar\omega$ is. The identification $\mathcal{K} \simeq
  \mathrm{fib}(\bar\alpha)$ follows from the octahedral axiom applied
  to $\cO_Y \to e^\perp \to V$.
\end{proof}

\begin{proposition}\label{prop:comparison}
  Let $X$ be a derived Artin stack, $\cL$ a line bundle on $X$ and $s
  \in \mathbb{Z}$.
  \begin{enumerate}
    \item Suppose that $\cL$ is trivial on an \'etale cover of $X$.
      Then $\Phi$ restricts to an equivalence from the subspace of
      $s$-shifted contact forms in the sense of
      Definition~\ref{def:berktav} to the subspace of $s$-shifted
      symplectic weight-one closed $2$-forms on $\widetilde X$. In
      particular, Definitions~\ref{def:contact} and~\ref{def:berktav}
      define the same notion of $s$-shifted contact structure, for
      every $s \in \mathbb{Z}$.
    \item For a morphism $g \colon Y \to X$, null-homotopies of
      $g^*\alpha$ in $\mathcal{A}^1(Y, g^*\cL; s)$ correspond under
      $\Phi$ to null-homotopies of $\tilde g^*\Phi(\alpha)$ in
      $\mathcal{A}^{2,\mathrm{cl}}(\widetilde Y, s)^{(1)}$.
  \end{enumerate}
\end{proposition}

\begin{proof}
  Part (2) is the naturality of $\Phi$. For (1), let $\omega =
  \Phi(\alpha)$ and apply Lemma~\ref{lem:reduction} with
  \[ Y = \widetilde X, \qquad V = \TT_{\widetilde X}, \qquad e = E,
  \qquad \text{and the underlying $2$-form } \omega_0. \]
  By Lemma~\ref{lem:exact}, $\lambda = \pm\iota_E\omega_0 =
  p^\sharp\alpha$. Let $\mathcal{K}$ be the complex of
  Definition~\ref{def:berktav}. Under the tautological trivialization
  of $p^*\cL$, the sequences \eqref{eq:atiyah} and $\mathcal{K} \to
  \TT_X \to \cL[s]$ give the diagram
  \[
    \begin{tikzcd}
      \cO_{\widetilde X} \arrow{r}{e'} \arrow[equal]{d} & e^\perp
      \arrow{r} \arrow{d} & p^*\mathcal{K} \arrow{d} \\
      \cO_{\widetilde X} \arrow{r}{E} & \TT_{\widetilde X} \arrow{r}
      \arrow{d}{\lambda^\vee} & p^*\TT_X \arrow{d}{p^*(\alpha^\vee)} \\
      & \cO_{\widetilde X}[s] \arrow{r}{\sim} & p^*\cL[s],
    \end{tikzcd}
  \]
  in which the columns and the first two rows are cofiber sequences. Hence
  \[ V/e \simeq p^*\TT_X, \qquad \bar\alpha = p^*(\alpha^\vee),
    \qquad \mathrm{cofib}(e') \simeq p^*\mathcal{K}, \qquad \bar\omega
  = p^*(d_{dR}\alpha|_{\mathcal{K}}). \]
  The last identification may be checked on an \'etale cover $U \to
  X$ trivializing $\cL$. There $\widetilde X \times_X U = U \times
  \Gm$ and $p^\sharp\alpha = t\alpha$, and
  \[ d_{dR}(t\alpha)|_{\mathcal{K}} = \big(d_{dR}t \wedge \alpha +
  t\, d_{dR}\alpha\big)|_{\mathcal{K}} = t\, d_{dR}\alpha|_{\mathcal{K}}, \]
  because $d_{dR}t$ vanishes on $\TT_U$ and $\alpha$ vanishes on
  $\mathcal{K}$. By Lemma~\ref{lem:reduction}, $\omega$ is symplectic
  if and only if $p^*(d_{dR}\alpha|_{\mathcal{K}})$ is an
  equivalence. Since $p$ is a smooth surjection, this holds if and
  only if $d_{dR}\alpha|_{\mathcal{K}}$ is an equivalence.
\end{proof}

\begin{definition}[{\cite[Definition~2.39]{IzbudakBerktav2026},
  \cite[Definitions~3.13 and~3.15]{kib3}}]\label{def:legendrian-berktav}
  Let $(X,\cL,\alpha)$ be an $s$-shifted contact derived Artin stack
  in the sense of Definition~\ref{def:berktav}, let $\mathcal{K} =
  \mathrm{fib}(\alpha^\vee)$, and let $g \colon Y \to X$ be a
  morphism of derived Artin stacks. We say that $g$ is
  \bfem{isotropic} if it is equipped with a null-homotopy $h$ of
  $g^*\alpha$ in $\mathcal{A}^1(Y, g^*\cL; s)$, and that an isotropic
  structure is \bfem{Legendrian in the sense of Berktav} if it
  induces a fiber sequence
  \[
    \begin{tikzcd}
      \cO_Y \arrow{r} & \LL_{Y/X} \otimes g^*\cL[s-1] \arrow{r} & \TT_Y
    \end{tikzcd}
  \]
  in $\QCoh(Y)$.
\end{definition}

We recall how the fiber sequence is induced by $h$, following the
remark after \cite[Definition~2.39]{IzbudakBerktav2026} and
\cite[Observation~3.14 and Definition~3.15]{kib3}. Since $g^*\alpha$
is the composite $\TT_Y \to g^*\TT_X \to g^*\cL[s]$, a null-homotopy
$h$ of $g^*\alpha$ is the same as a lift $\rho_h$ of $dg$,
\[
  \begin{tikzcd}
    & g^*\mathcal{K} \arrow{d} \\
    \TT_Y \arrow{ur}{\rho_h} \arrow{r}{dg} & g^*\TT_X.
  \end{tikzcd}
\]
The fiber sequence is obtained locally. Let $U \to X$ be an \'etale
map on which $\cL$ is trivialized, as in
Definition~\ref{def:berktav}, and let $Y_U := Y \times_X U$, which is
\'etale over $Y$. On $Y_U$, $h$ is a path of $1$-forms from $0$ to
$g^*\alpha$, so $d_{dR}h$ is a path from $0$ to $g^*d_{dR}\alpha$,
and it gives a null-homotopy of the composite
\[
  \begin{tikzcd}
    \TT_Y \arrow{r}{\rho_h} & g^*\mathcal{K}
    \arrow{r}{g^*(d_{dR}\alpha|_{\mathcal{K}})} &[2.5em]
    g^*\mathcal{K}^\vee \otimes g^*\cL[s] \arrow{r}{\rho_h^\vee} &
    \LL_Y \otimes g^*\cL[s]
  \end{tikzcd}
\]
on $Y_U$. With $C := \mathrm{cofib}(\rho_h)$, this null-homotopy
induces a map on $Y_U$
\[
  \begin{tikzcd}
    C \arrow{r}{\Theta_h} & \LL_Y \otimes g^*\cL[s].
  \end{tikzcd}
\] Let $\TT_{Y/X} := \mathrm{fib}(\TT_Y \to g^*\TT_X)$, so that
$\LL_{Y/X} \simeq \TT_{Y/X}^\vee$. The octahedral axiom for $\TT_Y
\to g^*\mathcal{K} \to g^*\TT_X$ gives the diagram of cofiber sequences
\[
  \begin{tikzcd}
    \TT_Y \arrow{r}{\rho_h} \arrow[equal]{d} & g^*\mathcal{K}
    \arrow{r} \arrow{d} & C \arrow{d} \\
    \TT_Y \arrow{r}{dg} & g^*\TT_X \arrow{r} \arrow{d} & \TT_{Y/X}[1]
    \arrow{d} \\
    & g^*\cL[s] \arrow[equal]{r} & g^*\cL[s].
  \end{tikzcd}
\]
Dualizing the right column and tensoring with $g^*\cL[s]$ gives the
cofiber sequence in the bottom row of the diagram
\[
  \begin{tikzcd}
    & & \TT_Y \arrow{d}{\Theta_h^\vee \otimes g^*\cL[s]} \\
    \cO_Y \arrow{r} & \LL_{Y/X} \otimes g^*\cL[s-1] \arrow{r} &
    C^\vee \otimes g^*\cL[s],
  \end{tikzcd}
\]
and the dual of $\Theta_h$ gives the vertical map.
The fiber sequence of Definition~\ref{def:legendrian-berktav} is
induced by $h$ if and only if $\Theta_h^\vee \otimes g^*\cL[s]$ is an
equivalence on $Y_U$ for an \'etale cover $\{U \to X\}$ trivializing
$\cL$. This means that $\Theta_h$ is an equivalence there.
Equivalently, the sequence
\[
  \begin{tikzcd}
    \TT_Y \arrow{r}{\rho_h} & g^*\mathcal{K} \arrow{r} & \LL_Y \otimes g^*\cL[s]
  \end{tikzcd}
\]
is locally a fiber sequence, which is the formulation of
\cite[Definition~3.15]{kib3}. By
Theorem~\ref{thm:legendrian-comparison}, this condition does not
depend on the cover or on the trivializations. For instance, for the
subscheme $L$ of Example~\ref{ex:legendrian}(1) with $h = 0$, the
sequence is $\TT_L \to i^*\mathcal{K} \to \LL_L$, and it is a fiber
sequence. In fact, along $L$ the bundle $i^*\mathcal{K}$ is spanned
by the $\partial_{x_i}$, which span $\TT_L$, and the
$\partial_{y_i}$, which $d_{dR}\alpha$ pairs perfectly with the
$\partial_{x_i}$.

To compare the two notions of Legendrian structure, we need a
relative form of Lemma~\ref{lem:reduction}.

\begin{lemma}[Lagrangian reduction]\label{lem:lagrangian-reduction}
  Let $Y$, $V$, $e$, $\omega$, $\lambda$, $e^\perp$, $e'$,
  $\mathcal{K}$ and $\bar\omega$ be as in Lemma~\ref{lem:reduction},
  and suppose that $\omega^\flat$ is an equivalence. Let $W \in
  \QCoh(Y)$ be perfect, $g \colon W \to V$ a map, $e_W \colon \cO_Y
  \to W$ a map with a homotopy $g \circ e_W \simeq e$, and $\eta$ a
  null-homotopy of the $2$-form $g^*\omega = \omega \circ \Lambda^2 g$ on $W$.
  \begin{enumerate}
    \item Contracting $\eta$ with $e_W$ gives a null-homotopy of
      $\lambda^\vee \circ g$, hence a lift $g' \colon W \to e^\perp$
      of $g$. There is a homotopy $g' \circ e_W \simeq e'$.
      Consequently $g'$ induces a map $\bar g \colon W/e_W \to
      \mathcal{K}$, where $W/e_W := \mathrm{cofib}(e_W)$.
    \item The null-homotopy $\eta$ descends to a null-homotopy
      $\bar\eta$ of the $2$-form $\bar g^*\bar\omega$ on $W/e_W$.
    \item The sequence
      \[
        \begin{tikzcd}
          W \arrow{r}{g} & V \arrow{r}{g^\vee[s] \circ \omega^\flat}
          &[2em] W^\vee[s],
        \end{tikzcd}
      \]
      with the null-homotopy of the composite given by $\eta$, is a
      fiber sequence if and only if the sequence
      \[
        \begin{tikzcd}
          W/e_W \arrow{r}{\bar g} & \mathcal{K} \arrow{r}{\bar
          g^\vee[s] \circ \bar\omega} &[2em] (W/e_W)^\vee[s],
        \end{tikzcd}
      \]
      with the null-homotopy given by $\bar\eta$, is a fiber sequence.
  \end{enumerate}
\end{lemma}

\begin{proof}\mbox{}
  \begin{enumerate}[wide]
    \item By the antisymmetry of $\omega$ and the homotopy $g \circ
      e_W \simeq e$,
      \[ \lambda^\vee \circ g \simeq \pm\, e^\vee[s] \circ
      \omega^\flat \circ g \simeq \pm\, e_W^\vee[s] \circ (g^*\omega)^\flat. \]
      The null-homotopy $\eta$ of $g^*\omega$ gives a null-homotopy
      of $(g^*\omega)^\flat$, and this yields the lift $g'$. Both $g'
      \circ e_W$ and $e'$ are lifts of $e$ through $e^\perp$, so they
      correspond to null-homotopies of $\lambda^\vee \circ e \simeq
      \omega(e,e)$. These null-homotopies are the restrictions of
      $\eta$ and of the canonical null-homotopy of $\omega \circ
      \Lambda^2 e$ along maps out of $\Lambda^2\cO_Y \simeq 0$. They
      agree because $\mathrm{Map}(\Lambda^2\cO_Y, \cO_Y[s])$ is
      contractible, so $g' \circ e_W \simeq e'$ and $g'$ induces a
      map of cofiber sequences
      \begin{equation}\label{eq:reduction-map}
        \begin{tikzcd}
          \cO_Y \arrow{r}{e_W} \arrow[equal]{d} & W \arrow{r}
          \arrow{d}{g'} & W/e_W \arrow{d}{\bar g} \\
          \cO_Y \arrow{r}{e'} & e^\perp \arrow{r} & \mathcal{K}.
        \end{tikzcd}
      \end{equation}

    \item Since $\Lambda^2\cO_Y \simeq 0$, the cofiber sequence
      $\cO_Y \to W \to W/e_W$ induces a cofiber sequence
      \[
        \begin{tikzcd}
          W/e_W \arrow{r}{e_W \wedge (-)} &[1.5em] \Lambda^2 W
          \arrow{r} & \Lambda^2(W/e_W).
        \end{tikzcd}
      \]
      A $2$-form on $W/e_W$ is therefore a $2$-form on $W$ together
      with a null-homotopy of its contraction with $e_W$. The same
      holds for null-homotopies of $2$-forms. The $2$-form
      $\bar\omega$ on $\mathcal{K}$ is obtained in this way from
      $\omega|_{e^\perp}$ and the null-homotopy of
      $\lambda|_{e^\perp}$ that defines $e^\perp$. It follows that
      $\bar g^*\bar\omega$ corresponds to $g^*\omega$ together with
      the null-homotopy of $\iota_{e_W}g^*\omega$ that defines $g'$,
      and this null-homotopy is the contraction of $\eta$ with $e_W$.
      So $\eta$ and its contraction with $e_W$ descend to a
      null-homotopy $\bar\eta$ of $\bar g^*\bar\omega$.

    \item Let $C := \mathrm{cofib}(g)$ and $\bar C :=
      \mathrm{cofib}(\bar g)$, and let $\Theta \colon C \to
      W^\vee[s]$ and $\bar\Theta \colon \bar C \to (W/e_W)^\vee[s]$
      be the maps induced by $g^\vee[s] \circ \omega^\flat$ and
      $\eta$, and by $\bar g^\vee[s] \circ \bar\omega$ and
      $\bar\eta$. The first condition says that $\Theta$ is an
      equivalence and the second that $\bar\Theta$ is, so we compare
      the two maps. The diagram \eqref{eq:reduction-map} identifies
      $\mathrm{cofib}(g')$ with $\bar C$, and the octahedral axiom
      for $W \to e^\perp \to V$ and the dual of $\cO_Y \to W \to
      W/e_W$ then give the rows of the diagram
      \[
        \begin{tikzcd}
          \bar C \arrow{r} \arrow{d}{\bar\Theta} & C \arrow{r}
          \arrow{d}{\Theta} & \cO_Y[s] \arrow{d}{\pm 1} \\
          (W/e_W)^\vee[s] \arrow{r} & W^\vee[s]
          \arrow{r}{e_W^\vee[s]} & \cO_Y[s],
        \end{tikzcd}
      \]
      where the upper right map is induced by $\lambda^\vee$ and the
      null-homotopy of $\lambda^\vee \circ g$ from (1). The right
      square commutes because $e_W^\vee[s] \circ g^\vee[s] \circ
      \omega^\flat \simeq \pm\lambda^\vee$, with both null-homotopies
      coming from $\eta$. For the left square, $\omega^\flat$ carries
      $e^\perp$ into $(V/e)^\vee[s]$, and by the construction of
      $\bar\omega$ the square
      \[
        \begin{tikzcd}
          e^\perp \arrow{r} \arrow{d} & (V/e)^\vee[s] \arrow{d} \\
          \mathcal{K} \arrow{r}{\bar\omega} & \mathcal{K}^\vee[s]
        \end{tikzcd}
      \]
      commutes. Since $W/e_W \xrightarrow{\bar g} \mathcal{K} \to
      V/e$ is the map induced by $g$, dualizing shows that the left
      square commutes, and the null-homotopies agree by (2). By the
      five lemma, $\Theta$ is an equivalence if and only if
      $\bar\Theta$ is. \qedhere
  \end{enumerate}
\end{proof}

\begin{theorem}\label{thm:legendrian-comparison}
  Let $(X,\cL,\alpha)$ be an $s$-shifted contact derived Artin stack
  such that $\cL$ is trivial on an \'etale cover of $X$. Let
  $(\widetilde X,\omega)$ be its symplectification, with $\omega =
  \Phi(\alpha)$, and let $g \colon Y \to X$ be a morphism of derived
  Artin stacks. For an isotropic structure $h$ on $g$ in the sense of
  Definition~\ref{def:legendrian-berktav}, let $\tilde h$ be the
  corresponding weight-one isotropic structure on $\tilde g \colon
  \widetilde Y \to \widetilde X$ given by
  Proposition~\ref{prop:comparison}(2). Then $h$ is Legendrian in the
  sense of Berktav (Definition~\ref{def:legendrian-berktav}) if and
  only if $\tilde h$ is Lagrangian, that is, a Legendrian structure
  on $g$ in the sense of Definition~\ref{def:legendrian}.
  Consequently, $h \mapsto \tilde h$ identifies the spaces of
  Legendrian structures on $g$ in the two senses, for every $s \in \mathbb{Z}$.
\end{theorem}

\begin{proof}
  Let us apply Lemma~\ref{lem:lagrangian-reduction} on $\widetilde
  Y$, and then identify its reduced sequence with the pullback of the
  sequence described after Definition~\ref{def:legendrian-berktav}.

  \bfem{Step 1.} Write $p_Y \colon \widetilde Y \to Y$ and $p_X
  \colon \widetilde X \to X$, so that $p_X \circ \tilde g = g \circ
  p_Y$. We apply Lemma~\ref{lem:lagrangian-reduction} on $\widetilde Y$ with
  \begin{align*}
    V &:= \tilde g^*\TT_{\widetilde X}, \text{ with the pullback of }
    \omega_0, & e &:= \tilde g^*E_{\widetilde X}, \\
    W &:= \TT_{\widetilde Y}, \text{ with } g := d\tilde g, & e_W &:=
    E_{\widetilde Y},
  \end{align*}
  and with $\eta$ the null-homotopy of $\tilde g^*\omega_0$
  underlying $\tilde h$. The homotopy $d\tilde g \circ E_{\widetilde
  Y} \simeq \tilde g^*E_{\widetilde X}$ holds because $\tilde g$ is
  $\Gm$-equivariant. By \cite[Definition~2.8]{PTVV}, in the
  equivalent formulation of the square \eqref{eq:lagrangian-square},
  $\tilde h$ is Lagrangian if and only if the first sequence of
  Lemma~\ref{lem:lagrangian-reduction}(3) is a fiber sequence.

  \medskip
  \bfem{Step 2.} Let $U \to X$ be an \'etale map trivializing $\cL$.
  Over $Y_U$, and under the tautological trivialization of
  $p_Y^*g^*\cL$, the objects of the second sequence of
  Lemma~\ref{lem:lagrangian-reduction}(3) are
  \begin{align*}
    W/e_W &\simeq p_Y^*\TT_Y && \text{by the Atiyah sequence of }
    \widetilde Y \to Y, \\
    \mathrm{cofib}(e') &\simeq p_Y^*g^*\mathcal{K}, \quad \bar\omega
    \simeq p_Y^*g^*(d_{dR}\alpha|_{\mathcal{K}}) && \text{as in the
    proof of Proposition~\ref{prop:comparison}}, \\
    \bar g &\simeq p_Y^*\rho_h && \text{by Lemma~\ref{lem:exact}}, \\
    \bar\eta &\simeq t\, d_{dR}h|_{\TT_Y}.
  \end{align*}
  The first line is \cite[Corollary~3.5]{IzbudakBerktav2026}. For the
  third, Lemma~\ref{lem:exact} for $Y$, applied to paths of forms,
  shows that contracting $\eta$ with $E_{\widetilde Y}$ gives the
  horizontal null-homotopy $p_Y^\sharp h$ of $\tilde
  g^*(\iota_E\omega_0)$, so the lift $g'$ of
  Lemma~\ref{lem:lagrangian-reduction}(1) is $p_Y^*\rho_h$. For the
  last, the trivialization of $\cL$ over $U$ gives the pullback square
  \[
    \begin{tikzcd}
      Y_U \times \Gm \arrow{r} \arrow{d} & \widetilde Y \arrow{d}{p_Y} \\
      Y_U \arrow{r} & Y,
    \end{tikzcd}
  \]
  in which $p_Y^\sharp h = t h$, so
  \[ \eta = d_{dR}(th) = d_{dR}t \wedge h + t\, d_{dR}h, \]
  and $d_{dR}t$ vanishes on the horizontal directions $\TT_{Y_U} =
  W/e_W$. The right-hand sides are the pullbacks along $p_Y$ of the
  objects of the sequence $\TT_Y \to g^*\mathcal{K} \to \LL_Y \otimes
  g^*\cL[s]$ described after Definition~\ref{def:legendrian-berktav}.
  Since $p_Y$ is a smooth surjection, one sequence is a fiber
  sequence if and only if the other is. Since the first sequence of
  Lemma~\ref{lem:lagrangian-reduction}(3) involves no choices, the
  condition in Definition~\ref{def:legendrian-berktav} does not
  depend on the cover or on the trivializations. The last statement
  follows because $h \mapsto \tilde h$ is an equivalence of spaces of
  isotropic structures by Proposition~\ref{prop:comparison}(2).
\end{proof}

\section{Twisted relative forms over
\texorpdfstring{$B\Gm$}{BGm}}\label{sec:twisted}

Let $S$ be a derived stack with a morphism $S \to B\Gm$, and write
$\mathcal{M}_S$ for the pullback of $\mathcal{M}$. In fact,
everything in this section holds, with the same proofs, for any line
bundle on $S$ in place of $\mathcal{M}_S$. We use this generality for
the powers $\mathcal{M}^{\otimes w}$ in Lemma~\ref{lem:weights}.
Recall that the sheaf $\DR_S(X)$ of \cite[Definition~B.11.4]{CHS2025}
is a sheaf of $\cO_S$-modules whose mixed structure is
$\cO_S$-linear, so the following definition makes sense.

\begin{definition}\label{def:twisted}
  For $X \in \dSt_S$, put $\DR_S(X;\mathcal{M}) := \DR_S(X)
  \otimes_{\cO_S} \mathcal{M}_S$, a sheaf of graded mixed $\cO_S$-modules, and
  \[ \mathcal{A}^{p,\mathrm{cl}}_S(X,s;\mathcal{M}) :=
    \mathrm{Map}_{\mathrm{GMMod}_{\K}}\big(\K(p)[-p-s],
  \Gamma_S\DR_S(X;\mathcal{M})\big). \]
  Likewise, the space $\mathcal{A}^{p}_S(X,s;\mathcal{M})$ is defined
  using underlying graded modules. With these forms in place of
  untwisted ones, an \bfem{$\mathcal{M}$-twisted $s$-presymplectic
  $S$-stack} is an Artin $S$-stack $X$ with $\omega \in
  \mathcal{A}^{2,\mathrm{cl}}_S(X,s;\mathcal{M})$, and it is
  \bfem{$s$-symplectic} if the underlying form induces an equivalence
  $\TT_{X/S} \to \LL_{X/S} \otimes \mathcal{M}_X[s]$, where
  $\mathcal{M}_X$ is the pullback of $\mathcal{M}$ to $X$.

  Isotropic and Lagrangian structures, and Lagrangian
  correspondences, are defined by the twisted forms of the squares
  \eqref{eq:lagrangian-square} and
  \eqref{eq:lagrangian-correspondence}. For instance, an isotropic
  structure on $g \colon Y \to X$ is a path from $0$ to $g^*\omega$
  in $\mathcal{A}^{2,\mathrm{cl}}_S(Y,s;\mathcal{M})$, and it is
  \bfem{Lagrangian} if the square
  \[
    \begin{tikzcd}
      \TT_{Y/S} \arrow{r} \arrow{d} & g^*\TT_{X/S} \arrow{d} \\
      0 \arrow{r} & \LL_{Y/S} \otimes \mathcal{M}_Y[s]
    \end{tikzcd}
  \]
  given by the isotropic structure is cartesian. Here the right
  vertical map is induced by the underlying map of $\omega$.
\end{definition}

For instance, if $S \to B\Gm$ factors through $q$, then
$\mathcal{M}_S$ is trivial and twisted forms are the relative forms
of Section~\ref{sec:recollections}. For $S = B\Gm$, on the other
hand, Definition~\ref{def:weights} says that
$\mathcal{A}^{p,\mathrm{cl}}_{B\Gm}(X,s;\mathcal{M}) =
\mathcal{A}^{p,\mathrm{cl}}(\widetilde X,s)^{(1)}$.
The properties of relative forms in Proposition~\ref{prop:CHS-forms}
carry over to twisted forms.

\begin{lemma}\label{lem:twisted}
  Let $S \to B\Gm$ be as above.
  \begin{enumerate}
    \item For a commutative square of derived stacks
      \[
        \begin{tikzcd}
          X' \arrow{r}{\xi} \arrow{d} & X \arrow{d} \\
          S' \arrow{r}{\sigma} & S,
        \end{tikzcd}
      \]
      there is a natural map $\sigma^*\DR_S(X;\mathcal{M}) \to
      \DR_{S'}(X';\mathcal{M})$, which is an equivalence if the
      square is cartesian.
    \item If $f \colon X \to S$ is geometric, then the underlying
      graded module of $\Gamma_S\DR_S(X;\mathcal{M})$ and the spaces
      of twisted forms are
      \begin{align*}
        &\Gamma_X\big(\mathrm{Sym}_{\cO_X}(\LL_{X/S}[-1]) \otimes
        f^*\mathcal{M}_S\big), \\
        &\mathcal{A}^p_S(X,s;\mathcal{M}) \simeq
        \mathrm{Map}_{\QCoh(X)}\big(\cO_X[-s], \Lambda^p\LL_{X/S}
        \otimes f^*\mathcal{M}_S\big).
      \end{align*}
    \item The functors $X \mapsto
      \mathcal{A}^{p,\mathrm{cl}}_S(X,s;\mathcal{M})$ and $X \mapsto
      \mathcal{A}^{p}_S(X,s;\mathcal{M})$ send colimits in $\dSt_S$
      to limits. They are therefore represented by objects
      $\mathcal{A}^{p,\mathrm{cl}}_S(s;\mathcal{M})$ and
      $\mathcal{A}^{p}_S(s;\mathcal{M})$ of $\dSt_S$. Closed twisted
      forms are functorial in commutative squares as in (1), and for
      $\sigma \colon S' \to S$ over $B\Gm$ there are canonical maps
      $\sigma^*\mathcal{A}^{p,\mathrm{cl}}_S(s;\mathcal{M}) \to
      \mathcal{A}^{p,\mathrm{cl}}_{S'}(s;\mathcal{M})$.
    \item There is a functor
      $((\dSt^{\mathrm{geom}}_S)_{/\mathcal{A}^p_S(s;\mathcal{M})})^{\mathrm{op}}
      \to \mathcal{Q}^{\Delta^1}$, with $\mathcal{Q}$ as in
      Appendix~\ref{app:chs}, sending $(X,\omega)$ to the map
      \vspace{-0.21in}

      \[
        \begin{tikzcd}
          \cO_X \arrow{r}{\omega} & (\otimes^p \LL_{X/S}) \otimes
          f^*\mathcal{M}_S[s].
        \end{tikzcd}
      \]
    \item There is a pullback square in $\dSt_S$, natural in $S$,
      \[
        \begin{tikzcd}
          \mathcal{A}^{2,\mathrm{cl}}_S(s-1;\mathcal{M}) \arrow{r}
          \arrow{d} & S \arrow{d}{0} \\
          S \arrow{r}{0} & \mathcal{A}^{2,\mathrm{cl}}_S(s;\mathcal{M}).
        \end{tikzcd}
      \]
    \item $\mathcal{A}^{2,\mathrm{cl}}_S(S,s;\mathcal{M}) \simeq *$,
      so $S$ has a unique $\mathcal{M}$-twisted $s$-presymplectic
      structure. It is symplectic.
  \end{enumerate}
\end{lemma}

\begin{proof}
  The proof tensors the statements of Section~\ref{sec:recollections}
  with $\mathcal{M}_S$, using two properties of this invertible
  sheaf. First, since the mixed structure of $\DR_S(X)$ is
  $\cO_S$-linear, $- \otimes_{\cO_S} \mathcal{M}_S$ is an
  autoequivalence of sheaves of graded mixed $\cO_S$-modules, and it
  preserves limits and colimits. Second, since $\mathcal{M}_S$ is
  dualizable, it satisfies the projection formula
  \begin{equation}\label{eq:twisted-projection}
    g_*(\mathcal{F} \otimes g^*\mathcal{M}_S) \simeq g_*\mathcal{F}
    \otimes \mathcal{M}_S
  \end{equation}
  for every morphism $g$ to $S$ and every quasi-coherent sheaf
  $\mathcal{F}$ on its source.

\begin{enumerate}[wide]\itemsep=0.2cm
    \item Proposition~\ref{prop:CHS-forms}(1) gives a natural map
      $\sigma^*\DR_S(X) \to \DR_{S'}(X')$, which is an equivalence
      for cartesian squares. The composite $S' \to S \to B\Gm$ gives
      $\sigma^*\mathcal{M}_S \simeq \mathcal{M}_{S'}$, and tensoring
      the map with it gives
      \[
        \begin{tikzcd}
          \sigma^*\DR_S(X;\mathcal{M}) \simeq \sigma^*\DR_S(X)
          \otimes \sigma^*\mathcal{M}_S \arrow{r} & \DR_{S'}(X')
          \otimes \mathcal{M}_{S'} = \DR_{S'}(X';\mathcal{M}),
        \end{tikzcd}
      \]
      which is again an equivalence for cartesian squares.

    \item We compute at points. Let $p \colon \Spec B \to S$ be a
      point, with base change $f_p \colon X_p \to \Spec B$. The value
      of $\DR_S(X;\mathcal{M})$ at $p$ is $\DR_B(X_p) \otimes_B
      p^*\mathcal{M}_S$, and by Proposition~\ref{prop:CHS-forms}(2)
      the underlying graded algebra of $\DR_B(X_p)$ is
      $f_{p,*}\mathrm{Sym}(\LL_{X_p/B}[-1])$. Hence, by
      \eqref{eq:twisted-projection} for $f_p$ and the invertible
      $B$-module $p^*\mathcal{M}_S$,
      \[ f_{p,*}\mathrm{Sym}(\LL_{X_p/B}[-1]) \otimes_B
        p^*\mathcal{M}_S \simeq
        f_{p,*}\big(\mathrm{Sym}(\LL_{X_p/B}[-1]) \otimes
      \mathcal{M}_{X_p}\big). \]
      These identifications are compatible with base change along
      maps of points, by Proposition~\ref{prop:CHS-forms}(1). Passing
      to the limit over points gives the first formula. The
      weight-$p$ summand of $\mathrm{Sym}(\LL_{X/S}[-1])$ is
      $\Lambda^p\LL_{X/S}[-p]$, so mapping out of $\K(p)[-p-s]$ gives
      the second formula, as in Proposition~\ref{prop:CHS-forms}(2).

    \item The functor $X \mapsto \DR_S(X)$ sends colimits in $\dSt_S$
      to limits, since it is defined as a limit over the affine
      derived schemes over $X$ (Section~\ref{subsec:relative-forms}).
      Since the functors $- \otimes \mathcal{M}_S$, $\Gamma_S$ and
      $\mathrm{Map}(\K(p)[-p-s], -)$ preserve limits, both functors
      of twisted forms send colimits to limits. Moreover, every
      derived stack over $S$ is a colimit of affine derived schemes
      over $S$, so such a functor is determined by its restriction to
      affine derived schemes over $S$. This restriction is a sheaf,
      since the functor sends colimits to limits, and it is the
      representing object. Finally, the functoriality and the maps
      $\sigma^*\mathcal{A}^{p,\mathrm{cl}}_S(s;\mathcal{M}) \to
      \mathcal{A}^{p,\mathrm{cl}}_{S'}(s;\mathcal{M})$ come from the
      maps of (1), in the same way as in Proposition~\ref{prop:CHS-forms}(3).

    \item By (2), the functor $X \mapsto
      \mathcal{A}^p_S(X,s;\mathcal{M})$ on geometric $S$-stacks is
      the composite of the functor $X \mapsto (X, \Lambda^p\LL_{X/S}
      \otimes f^*\mathcal{M}_S[s])$ to $\mathcal{Q}$ with
      $\mathrm{Map}_{\mathcal{Q}}(\cO, -)$. The construction of
      Proposition~\ref{prop:forms-functor} from the corresponding
      composite for untwisted forms applies to it and gives the functor.

    \item For every $S$-stack $X$, applying
      $\mathrm{Map}_{\mathrm{GMMod}_{\K}}(-,
      \Gamma_S\DR_S(X;\mathcal{M}))$ to the pushout square of graded
      mixed $\K$-modules
      \begin{equation}\label{eq:GMM-pushout}
        \begin{tikzcd}
          \K(2)[-2-s] \arrow{r} \arrow{d} & 0 \arrow{d} \\
          0 \arrow{r} & \K(2)[-1-s]
        \end{tikzcd}
      \end{equation}
      gives a pullback square of spaces
      \[
        \begin{tikzcd}
          \mathcal{A}^{2,\mathrm{cl}}_S(X,s-1;\mathcal{M}) \arrow{r}
          \arrow{d} & * \arrow{d}{0} \\
          * \arrow{r}{0} & \mathcal{A}^{2,\mathrm{cl}}_S(X,s;\mathcal{M}).
        \end{tikzcd}
      \]
      These squares are natural in $X$, and in $S$ by (1). By (3)
      they are the pullback square of the representing objects
      evaluated at $X$. This proves (5), by the argument for the
      square \eqref{eq:CHS-delooping} with
      $\Gamma_S\DR_S(X;\mathcal{M})$ in place of $\DR(X/S)$.

    \item The value of $\DR_S(S)$ at a point $\Spec B \to S$ is
      $\DR_B(\Spec B) = B$, in weight zero with zero mixed structure,
      so $\Gamma_S\DR_S(S;\mathcal{M}) \simeq \Gamma_S\mathcal{M}_S$
      is concentrated in weight zero. Since $\K(2)[-2-s]$ is
      concentrated in weight two, we get
      \[ \mathcal{A}^{2,\mathrm{cl}}_S(S,s;\mathcal{M}) =
        \mathrm{Map}_{\mathrm{GMMod}_{\K}}\big(\K(2)[-2-s],
      \Gamma_S\mathcal{M}_S\big) \simeq *. \]
      The unique structure is symplectic, since $\TT_{S/S} \simeq 0
      \simeq \LL_{S/S}$. \qedhere
  \end{enumerate}
\end{proof}

Over $S = B\Gm$, twisted forms describe contact geometry.

\begin{proposition}\label{prop:relative}
  Let $X$ be a derived Artin stack with a line bundle $\cL$,
  classified by $X \to B\Gm$, with principal $\Gm$-bundle $\widetilde X$.
\begin{enumerate}\itemsep=0.2cm
    \item $\mathcal{M}$-twisted $s$-presymplectic structures on the
      $B\Gm$-stack $X$ are the same as weight-one closed $2$-forms on
      $\widetilde X$. Such a structure is symplectic if and only if
      the weight-one form is symplectic. Thus $\mathcal{M}$-twisted
      $s$-symplectic $B\Gm$-stacks are the same as $s$-shifted
      contact derived Artin stacks in the sense of Definition~\ref{def:contact}.
    \item Let $g \colon Y \to X$ be a morphism of derived Artin
      stacks, regarded as a morphism over $B\Gm$. Then
      $\mathcal{M}$-twisted isotropic structures on $g$ relative to
      $B\Gm$ are the same as weight-one isotropic structures on
      $\tilde g$. Under this identification, $\mathcal{M}$-twisted
      Lagrangian structures are the same as Legendrian structures on
      $g$ (Definition~\ref{def:legendrian}).
    \item For contact stacks $X$ and $Y$, the fiber product
      \[ X \circledast Y := X \times_{B\Gm} Y \simeq [(\widetilde X
      \times \widetilde Y)/\Gm], \]
      with $\Gm$ acting diagonally, is an $s$-shifted contact stack
      for the twisted form $\pi_X^*\omega_X + \pi_Y^*\omega_Y$. We
      call it the \bfem{contact product}. Its symplectification is
      $(\widetilde X \times \widetilde Y, \omega_{\widetilde X}
      \boxplus \omega_{\widetilde Y})$, and the unit for
      $\circledast$ is $B\Gm$ with its unique structure, whose
      symplectification is the point.
  \end{enumerate}
\end{proposition}

\begin{proof}\mbox{}
\begin{enumerate}[wide]\itemsep=0.1cm
    \item The first statement is the definition. For non-degeneracy,
      let $\omega^0$ be the underlying twisted $2$-form and
      $\widetilde\omega_0$ the underlying weight-one $2$-form on
      $\widetilde X$. By Lemma~\ref{lem:weights}(1), under
      $p^*\TT_{X/B\Gm} \simeq \TT_{\widetilde X}$,
      \[ \widetilde\omega_0^\flat = \tau \circ p^*(\omega^{0,\flat}). \]
      So $\widetilde\omega_0^\flat$ is an equivalence if and only if
      $p^*(\omega^{0,\flat})$ is, and since $p$ is a smooth
      surjection, this holds if and only if $\omega^{0,\flat}$ is an
      equivalence.

    \item By (1) applied to $Y$ and Lemma~\ref{lem:weights}(2), paths
      in twisted forms on $Y$ are paths in weight-one forms on
      $\widetilde Y$. It remains to compare the Lagrangian
      conditions. The twisted one is a condition on a map of perfect
      complexes on $Y$ relative to $B\Gm$, and by
      Proposition~\ref{prop:CHS-forms}(1) its pullback along the
      smooth surjection $\widetilde Y \to Y$ is the Lagrangian
      condition for $\tilde g$.

    \item The square
      \[
        \begin{tikzcd}
          \widetilde X \times \widetilde Y \arrow{r} \arrow{d} &
          \mathrm{pt} \arrow{d}{q} \\
          X \times_{B\Gm} Y \arrow{r} & B\Gm
        \end{tikzcd}
      \]
      is a pullback, compatibly with the $\Gm$-actions, since $X
      \times_{B\Gm} Y \times_{B\Gm} \mathrm{pt} \simeq \widetilde X
      \times \widetilde Y$. Descent along $q$ then gives $X
      \times_{B\Gm} Y \simeq [(\widetilde X \times \widetilde
      Y)/\Gm]$. The sum of the pulled-back twisted forms is symplectic because
      \[ \TT_{(X \times_{B\Gm} Y)/B\Gm} \simeq \pi_X^*\TT_{X/B\Gm}
      \oplus \pi_Y^*\TT_{Y/B\Gm} \]
      and the form is the sum of the pulled-back forms. The statement
      about symplectifications follows from (1) and
      Lemma~\ref{lem:weights}(2), and the statement about the unit is
      Lemma~\ref{lem:twisted}(6). \qedhere
  \end{enumerate}
\end{proof}

\section{The \texorpdfstring{$(\infty,n)$}{(infinity,n)}-categories
of Legendrian correspondences}\label{sec:Leg}

In this section, we construct the $(\infty,n)$-categories $\Leg_n^s$
and prove Theorem~\ref{thm:A}. We work over a derived stack $S \to
B\Gm$ and follow \cite{CHS2025}, with twisted forms in place of forms.

\begin{definition}\label{def:Leg}
  For $S \to B\Gm$, let
  \[ \PreSymp^{\mathcal{M}}_{S,s} := \dSt^{\mathrm{Art}}_S
  \times_{\dSt_S} (\dSt_S)_{/\mathcal{A}^{2,\mathrm{cl}}_S(s;\mathcal{M})} \]
  be the $\infty$-category of $\mathcal{M}$-twisted $s$-presymplectic
  $S$-stacks, as in \cite[Definition~2.3.2]{CHS2025}. Let
  $\Leg_n^{S,s} \subseteq \Span_n(\PreSymp^{\mathcal{M}}_{S,s})$ be
  the sub-$(\infty,n)$-category of iterated spans satisfying the
  twisted Lagrangian condition, defined as in
  \cite[Definition~2.9.7]{CHS2025}. We write
  \[ \Leg_n^{s} := \Leg_n^{B\Gm,s}, \qquad \Leg_n^{X/B\Gm,s} :=
  \Leg_{n+1}^{B\Gm,s}(B\Gm, X). \]
\end{definition}

For instance, for $S = \mathrm{pt}$ with the map $q$, the sheaf
$\mathcal{M}_S$ is trivial and $\Leg_n^{\mathrm{pt},s} \simeq
\Lag_n^{\mathrm{pt},s}$, while $S = B\Gm$ gives the
$(\infty,n)$-categories $\Leg_n^s$.

\begin{theorem}\label{thm:transfer}
  Let $S \to B\Gm$ be a derived stack over $B\Gm$, $s \in \mathbb{Z}$
  and $n \geq 1$. The statements of Theorem~\ref{thm:CHS} hold for
  $\mathcal{M}$-twisted forms, with $\PreSymp_{S,s}$, $\Lag_n^{S,s}$
  and $\Lag_n^{X/S,s}$ replaced by $\PreSymp^{\mathcal{M}}_{S,s}$,
  $\Leg_n^{S,s}$ and $\Leg_n^{X/S,s} := \Leg_{n+1}^{S,s}(S,X)$. We
  refer to the twisted form of Theorem~\ref{thm:CHS}(k) as
  Theorem~\ref{thm:transfer}(k).
\end{theorem}

\begin{proof}
  We go through the proof of Theorem~\ref{thm:CHS} in \cite{CHS2025},
  with twisted forms in place of forms. As explained at the end of
  Section~\ref{sec:recollections}, differential forms enter that
  proof in four places, and Lemma~\ref{lem:twisted} and
  Proposition~\ref{prop:twisted-qcoh-diagram} provide their twisted
  forms. We first treat twisted presymplectic stacks and the diagrams
  of quasi-coherent sheaves attached to their iterated spans (Steps 1
  and 2). We then prove (1) in Step 3, the symmetric monoidal
  structure and (2) in Step 4, (3) and (4) in Step 5, and (5) in Step 6.

  \bfem{Step 1: twisted presymplectic stacks.} By
  Lemma~\ref{lem:twisted}(3), a twisted $s$-presymplectic $S$-stack
  is a morphism $X \to \mathcal{A}^{2,\mathrm{cl}}_S(s;\mathcal{M})$
  in $\dSt_S$ with $X$ Artin over $S$, so
  $\PreSymp^{\mathcal{M}}_{S,s}$ is the pullback of
  Definition~\ref{def:Leg}. The slice
  $(\dSt_S)_{/\mathcal{A}^{2,\mathrm{cl}}_S(s;\mathcal{M})}$ has all
  limits. Its forgetful functor to $\dSt_S$ is a right fibration by
  the dual of \cite[Corollary~2.1.2.2]{LurieHTT}, and hence preserves
  weakly contractible limits by the dual of
  \cite[Proposition~4.4.2.9]{LurieHTT}. Artin $S$-stacks are closed
  under finite limits in $\dSt_S$ by
  Proposition~\ref{prop:dag-facts}(2). By the dual of
  \cite[Lemma~5.4.5.5]{LurieHTT}, $\PreSymp^{\mathcal{M}}_{S,s}$ has
  weakly contractible finite limits, in particular pullbacks, and
  they are computed in $\dSt_S$. For $\sigma \colon S' \to S$ over
  $B\Gm$, Lemma~\ref{lem:twisted}(1) and (3) give a base change
  functor $\sigma^* \colon \PreSymp^{\mathcal{M}}_{S,s} \to
  \PreSymp^{\mathcal{M}}_{S',s}$, given on stacks by pullback. For a
  pullback square with $\xi \colon X' \to X$ over $\sigma$,
  Lemma~\ref{lem:twisted}(2) and Proposition~\ref{prop:CHS-forms}(1)
  show that the underlying map of the pulled-back form is the composite
  \[
    \begin{tikzcd}
      \TT_{X'/S'} \simeq \xi^*\TT_{X/S} \arrow{r}{\xi^*\tilde\omega}
      & \xi^*\LL_{X/S} \otimes \mathcal{M}_{X'}[s] \simeq \LL_{X'/S'}
      \otimes \mathcal{M}_{X'}[s],
    \end{tikzcd}
  \]
  since $\xi^*\mathcal{M}_X \simeq \mathcal{M}_{X'}$.

  \medskip
  \bfem{Step 2: twisted Lagrangian conditions.} By
  Proposition~\ref{prop:twisted-qcoh-diagram}, an $n$-uple or
  $n$-fold span in $\PreSymp^{\mathcal{M}}_{S,s}$ with apex $X$ gives
  a diagram $\widetilde{\mathsf{Sp}}{}^n \to \QCoh(X)$ or
  $\widetilde{\mathsf{L}}{}^n \to \QCoh(X)$, whose values are the
  pullbacks of the relative tangent complexes of the vertices and of
  their twisted shifted cotangent complexes $\LL \otimes
  \mathcal{M}[s]$. We say that the span is \bfem{twisted Lagrangian}
  if this diagram is non-degenerate, and a map
  $\mathsf{S}^{\mathbf{k}} \to \PreSymp^{\mathcal{M}}_{S,s}$ is
  twisted Lagrangian if it is cartesian and its restrictions to the
  cubes of $\mathsf{L}^{\mathbf{k}}$ are twisted Lagrangian, as in
  Appendix~\ref{app:chs}. For $n = 1$ the diagram is the one
  displayed before Proposition~\ref{prop:nondeg-criteria}, with $\LL
  \otimes \mathcal{M}[s]$ in place of $\LL[s]$, and it is
  non-degenerate if and only if the twisted form of the square
  \eqref{eq:lagrangian-correspondence} is cartesian. These twisted
  Lagrangian maps span a full subcategory $\mathrm{LEG}^{+,S,s}_n$ of
  $\mathrm{SPAN}^{+}_{n}(\PreSymp^{\mathcal{M}}_{S,s})$. As in the
  untwisted case, these conditions are local on the base. In fact,
  equivalences and finite limits in quasi-coherent sheaves can be
  checked after pullback along points by
  Proposition~\ref{prop:dag-facts}(1), so by
  Proposition~\ref{prop:twisted-qcoh-diagram}(2) the twisted
  symplectic and Lagrangian conditions can be checked after base
  change along all points $\Spec A \to S$.

  \medskip
  \bfem{Step 3: composition and the proof of (1).} By
  Theorem~\ref{thm:spans}(1), the projection
  \[
    \begin{tikzcd}
      \mathrm{SPAN}^{+}_{n}(\PreSymp^{\mathcal{M}}_{S,s}) \arrow{r} &
      \Delta^{n,\mathrm{op}}
    \end{tikzcd}
  \]
  is a cocartesian fibration, and we show that its cocartesian
  pushforwards preserve $\mathrm{LEG}^{+,S,s}_n$. As recalled in
  Appendix~\ref{app:chs}, it suffices to consider the degeneracy
  $s^0$ and the inner face $d^1$ in one coordinate. For $s^0$, the
  diagram of the degenerate span is the diagram of the given span
  composed with $\mathrm{Tw}^r_!(\pi_j)$, so it is non-degenerate by
  Proposition~\ref{prop:composition-criteria}(1). For $d^1$, let $F
  \colon \mathsf{S}^{2,1,\ldots,1} \to \PreSymp^{\mathcal{M}}_{S,s}$
  be twisted Lagrangian and $\Phi$ its diagram. We check the three
  hypotheses of Proposition~\ref{prop:composition-criteria}(2).
  \begin{enumerate}
    \item[(b)] The restrictions of $\Phi$ to the cubes are
      non-degenerate because $F$ is twisted Lagrangian.
    \item[(c)] For $X \in \mathsf{S}^{2,1,\ldots,1}$, the value
      $\Phi(X)$ is the pullback of $\TT_{F(X)/S}$. Since $F$ is
      cartesian, $F(X)$ is the limit of $F$ on
      $\mathsf{L}^{2,1,\ldots,1}_{X/}$. By
      Proposition~\ref{prop:dag-facts}(3), the relative tangent
      complex takes finite limits of Artin $S$-stacks to limits, so
      $\Phi(X)$ is the limit of $\Phi$ on $\mathsf{L}^{2,1,\ldots,1}_{X/}$.
    \item[(a)] The square of (a) is obtained from a pullback square
      of tangent complexes, given by (c), by dualizing and tensoring
      with the invertible sheaf $\mathcal{M}_X[s]$. It is therefore a
      pushout square, and hence a pullback square, since $\QCoh(X)$ is stable.
  \end{enumerate}
  By Proposition~\ref{prop:composition-criteria}(2), the composite
  \[
    \begin{tikzcd}
      \mathsf{Sp}^n \arrow{r}{(d^1, \mathrm{id}, \ldots,
      \mathrm{id})} &[3.5em] \mathsf{S}^{2,1,\ldots,1} \arrow{r}{F} &
      \PreSymp^{\mathcal{M}}_{S,s}
    \end{tikzcd}
  \]
  is twisted Lagrangian, and hence $\mathrm{LEG}^{+,S,s}_n \to
  \Delta^{n,\mathrm{op}}$ is a cocartesian fibration. Since the
  twisted Lagrangian condition on a grid of spans is a condition on
  its cubes, the Segal conditions of
  $\mathrm{SPAN}^{+}_{n}(\PreSymp^{\mathcal{M}}_{S,s})$ restrict to
  it, and its underlying $n$-fold Segal space is $\Leg_n^{S,s}$. This
  Segal space is complete. In fact,
  $\Span_n(\PreSymp^{\mathcal{M}}_{S,s})$ is complete by
  Theorem~\ref{thm:spans}(1), and an equivalence of iterated spans
  induces an equivalence of their diagrams, so it preserves non-degeneracy.

  \medskip
  \bfem{Step 4: the symmetric monoidal structure and the proof of
  (2).} The symmetric monoidal structure comes from delooping. By
  Lemma~\ref{lem:twisted}(5), a twisted $s$-presymplectic $S$-stack
  with two maps to $S$ is an Artin $S$-stack together with two
  null-homotopies of its twisted form, and these form a loop at $0$
  in $\mathcal{A}^{2,\mathrm{cl}}_S(-,s;\mathcal{M})$, i.e., a
  twisted $(s-1)$-form, leading to an equivalence
  \[ (\PreSymp^{\mathcal{M}}_{S,s})_{/S,S} \simeq
  \PreSymp^{\mathcal{M}}_{S,s-1} \]
  over $(\dSt_S)_{/S,S} \simeq \dSt_S$, natural in $S$, under which
  the degenerate span of $(S,0)$ corresponds to $(S,0)$ by
  Lemma~\ref{lem:twisted}(6). By Theorem~\ref{thm:spans}(2) it gives
  an equivalence
  \[ \Span_n(\PreSymp^{\mathcal{M}}_{S,s})(S,S) \simeq
  \Span_{n-1}(\PreSymp^{\mathcal{M}}_{S,s-1}). \]
  Applying Theorem~\ref{thm:spans}(3) to the pointed
  $\infty$-categories $\PreSymp^{\mathcal{M}}_{S,s+k}$ with base
  point $S$, for $k \geq 0$, then gives a symmetric monoidal
  structure on $\Span_n(\PreSymp^{\mathcal{M}}_{S,s})$ with unit
  $(S,0)$. The group structure on
  $\mathcal{A}^{2,\mathrm{cl}}_S(s;\mathcal{M})$ given by the square
  of Lemma~\ref{lem:twisted}(5) is the addition of forms, since both
  come from the shift functor of graded mixed modules, as in the
  untwisted case (Section~\ref{subsec:recollection-lag}). The tensor
  product is therefore given on objects by
  \[ (X,\omega) \otimes (X',\omega') = (X \times_S X', \pi_X^*\omega
  + \pi_{X'}^*\omega'). \]

  We now show that this equivalence restricts to $\Leg_n^{S,s}(S,S)
  \simeq \Leg_{n-1}^{S,s-1}$, so that the symmetric monoidal
  structure restricts to $\Leg_n^{S,s}$. For $n = 1$, a span $S
  \leftarrow X \to S$ has $\TT_{S/S} \simeq 0$ at two corners of the
  twisted square \eqref{eq:lagrangian-correspondence}, so the square
  is cartesian if and only if the induced map $\TT_{X/S} \to
  (\LL_{X/S} \otimes \mathcal{M}_X[s])[-1] \simeq \LL_{X/S} \otimes
  \mathcal{M}_X[s-1]$ is an equivalence. This is the twisted
  $(s-1)$-symplectic condition for the $(s-1)$-form given by
  Lemma~\ref{lem:twisted}(5). For $n > 1$, by induction on $n$, it
  suffices to show the following. Let $F \colon \mathsf{Sp}^n \to
  \PreSymp^{\mathcal{M}}_{S,s}$ be an $n$-uple span from $S$ to $S$
  whose boundary is twisted Lagrangian, with apex $X$ and diagram
  $\Phi$. Then $F$ is twisted Lagrangian if and only if its
  restriction to $\{(0,1)\} \times \mathsf{Sp}^{n-1}$, regarded as an
  $(n-1)$-uple span in $\PreSymp^{\mathcal{M}}_{S,s-1}$, is twisted
  Lagrangian. The restrictions of $F$ to $\{(0,0)\} \times
  \mathsf{Sp}^{n-1}$ and $\{(1,1)\} \times \mathsf{Sp}^{n-1}$ are
  constant at $S$, and $\TT_{S/S} \simeq 0 \simeq \LL_{S/S}$, so
  $\Phi$ vanishes on the corresponding parts of
  $\widetilde{\mathsf{Sp}}{}^n$. By
  Proposition~\ref{prop:nondeg-criteria}(1), $F$ is twisted
  Lagrangian if and only if $\TT_{X/S}$ is the limit of the right Kan
  extension $\Psi$ on $\{-\infty\} \times
  \mathsf{Sp}^{n-1,\mathrm{op}}$. Since $\Phi$ vanishes over the two
  copies of $S$, this limit is the limit of the restriction of $\Phi$
  to $(0,1) \times \mathsf{Sp}^{n-1,\mathrm{op}}$ shifted by $[-1]$.
  The shift turns the twisted cotangent terms into
  \[ \big(\LL_{Y/S} \otimes \mathcal{M}_Y[s]\big)[-1] \simeq
  \LL_{Y/S} \otimes \mathcal{M}_Y[s-1], \]
  which are the twisted cotangent terms for the $(s-1)$-forms given
  by Lemma~\ref{lem:twisted}(5). By
  Proposition~\ref{prop:nondeg-criteria}(3) and the inductive
  hypothesis, this is the twisted Lagrangian condition on the
  restriction in $\PreSymp^{\mathcal{M}}_{S,s-1}$.

  \medskip
  \bfem{Step 5: adjoints, mapping categories and duals, and the
  proofs of (3) and (4).} By Theorem~\ref{thm:spans}(5),
  $1$-morphisms of $\Span_n(\PreSymp^{\mathcal{M}}_{S,s})$ have
  adjoints given by reversed spans. So for a twisted Lagrangian
  correspondence $X \xleftarrow{f} Y \xrightarrow{g} X'$ it suffices
  to show that the reversed span and the unit and counit of the
  adjunction are twisted Lagrangian. For the reversed span this is
  clear, since it has the transposed square, which is cartesian if
  and only if the original one is. For the unit and counit, denote by
  $\TT_X$, $\TT_{X'}$ and $\LL_Y$ the relative tangent and cotangent
  complexes over $S$, pulled back to the apexes of the spans below.
  Our hypothesis is that the square
  \[
    \begin{tikzcd}
      \TT_{Y} \arrow{r} \arrow{d} & \TT_{X'} \arrow{d} \\
      \TT_{X} \arrow{r} & \LL_{Y} \otimes \mathcal{M}_Y[s]
    \end{tikzcd}
  \]
  is cartesian. Consider first the unit, which is the $2$-fold span
  $X \leftarrow Y \to Y \times_{X'} Y$ over $X$ and $X$. Let $M :=
  \TT_X \times_{\LL_Y \otimes \mathcal{M}_Y[s]} \TT_X$. By
  Proposition~\ref{prop:nondeg-criteria}(2), the unit is twisted
  Lagrangian if the left square in the diagram
  \[
    \begin{tikzcd}
      \TT_{Y} \arrow{r} \arrow{d} & \TT_{Y \times_{X'} Y} \arrow{r}
      \arrow{d} & \TT_{Y} \arrow{d} \\
      \TT_{X} \arrow{r} & M \arrow{r} & \TT_{X}
    \end{tikzcd}
  \]
  is cartesian. To see this, consider the diagram
  \[
    \begin{tikzcd}
      \TT_{Y \times_{X'} Y} \arrow{r} \arrow{d} & \TT_{Y} \arrow{d} \\
      \TT_{Y} \arrow{r} \arrow{d} & \TT_{X'} \arrow{d} \\
      \TT_{X} \arrow{r} & \LL_{Y} \otimes \mathcal{M}_Y[s]
    \end{tikzcd}
  \]
  in which the upper square is cartesian by
  Proposition~\ref{prop:dag-facts}(3) and the lower square by
  hypothesis. The outer rectangle is therefore cartesian, and it
  factors through the square with corner $M$ as
  \[
    \begin{tikzcd}
      \TT_{Y \times_{X'} Y} \arrow{r} \arrow{d} & \TT_{Y} \arrow{d} \\
      M \arrow{r} \arrow{d} & \TT_{X} \arrow{d} \\
      \TT_{X} \arrow{r} & \LL_{Y} \otimes \mathcal{M}_Y[s],
    \end{tikzcd}
  \]
  whose lower square is cartesian by definition of $M$. Hence the
  upper square is cartesian. It is the right square of the first
  diagram, in which both horizontal composites are identities, so the
  outer square of that diagram is cartesian, and hence so is its left
  square. Likewise, the counit is twisted Lagrangian, and the same
  argument, with Proposition~\ref{prop:nondeg-criteria}(2), gives
  adjoints for $1$-morphisms in $\Leg_n^{S,s}(X,X')$. This proves the
  first part of (3).

  For (4), let $X$ be twisted $s$-symplectic and $Y, Z$ objects of
  $\Leg_n^{X/S,s} = \Leg_{n+1}^{S,s}(S,X)$. The $1$-morphism $Y
  \colon S \to X$ has the right adjoint $Y^{\mathrm{rev}}$, and we
  use a general fact about adjunctions. If $f \dashv g$ is an
  adjunction between $A$ and $B$ in an $(\infty,n+1)$-category
  $\mathcal{C}$ and $C$ is an object of $\mathcal{C}$, then
  composition with $f$ and $g$ gives an adjunction between
  $\mathcal{C}(C,A)$ and $\mathcal{C}(C,B)$, and hence
  \[ \mathcal{C}(C,B)(f\alpha, \beta) \simeq \mathcal{C}(C,A)(\alpha,
  g\beta). \]
  Taking $\mathcal{C} = \Leg_{n+1}^{S,s}$, $C = A = S$, $B = X$, $f =
  Y$, $\alpha = \mathrm{id}_S$ and $\beta = Z$, and using (2), we obtain
  \[ \Leg_n^{X/S,s}(Y,Z) \simeq \Leg_{n+1}^{S,s}(S,S)(S, Z \times_X
    Y^{\mathrm{rev}}) \simeq \Leg_n^{S,s-1}(S, Z \times_X
  Y^{\mathrm{rev}}) = \Leg_{n-1}^{Z \times_X Y^{\mathrm{rev}}/S,\, s-1}. \]
  By induction, the $(\infty,n)$-categories $\Leg_n^{X/S,s}$ have
  adjoints for all morphisms. Finally, by (2) the symmetric monoidal
  $(\infty,n)$-category $\Leg_n^{S,s}$ is the endomorphism category
  of $S$ in $\Leg_{n+1}^{S,s+1}$, so its objects have duals because
  $1$-morphisms of $\Leg_{n+1}^{S,s+1}$ have adjoints. Together with
  the adjoints for all morphisms this proves full dualizability, as
  for Theorem~\ref{thm:CHS}(3).

  \medskip
  \bfem{Step 6: the AKSZ construction and the proof of (5).} Let $T$
  be a $d$-preoriented $S$-stack with structure map $\phi$, $X$ a
  geometric $S$-stack with structure map $f$, and $Y := X \times_S
  T$. Tensoring the equivalence of
  Proposition~\ref{prop:CHS-forms}(4) with $\mathcal{M}_S$ and using
  \eqref{eq:twisted-projection} for $\phi$ gives
  \begin{equation}\label{eq:twisted-pullback-DR}
    \Gamma_S\big(\phi_*\cO_T \otimes_S \DR_S(X;\mathcal{M})\big)
    \simeq \Gamma_T \DR_T(Y; \phi^*\mathcal{M}_S).
  \end{equation}
  Composing with the map $[T] \otimes \mathrm{id} \colon \phi_*\cO_T
  \otimes \mathcal{M}_S \to \mathcal{M}_S[-d]$ and mapping out of
  $\K(p)[-p-s]$ gives the twisted pushforward
  \[
    \begin{tikzcd}
      \mathcal{A}^{p,\mathrm{cl}}_T(Y,s;\phi^*\mathcal{M}_S)
      \arrow{r}{\int_{[T]}} & \mathcal{A}^{p,\mathrm{cl}}_S(X,s-d;\mathcal{M}).
    \end{tikzcd}
  \]
  It is compatible with base change by Lemma~\ref{lem:twisted}(1),
  and as in Section~\ref{subsec:recollection-aksz} it can be
  described on underlying forms. If $X$ is an Artin $S$-stack and the
  underlying form of $\omega$ corresponds to $\omega'' \colon
  \Lambda^p\TT_{X/S} \to f^*\phi_*\cO_T \otimes \mathcal{M}_X[s]$,
  then by Lemma~\ref{lem:twisted}(2) the underlying form of
  $\int_{[T]}\omega$ corresponds to the composite
  \[
    \begin{tikzcd}
      \Lambda^p\TT_{X/S} \arrow{r}{\omega''} & f^*\phi_*\cO_T \otimes
      \mathcal{M}_X[s] \arrow{r}{f^*[T] \otimes \mathrm{id}} &[2em]
      \mathcal{M}_X[s-d].
    \end{tikzcd}
  \]
  This is the twisted form of \eqref{eq:CHS-pushforward}. Now let $X$
  be twisted $s$-symplectic and $T$ an $\cO$-compact $d$-oriented
  $S$-stack with $X^T_S$ Artin. The twisted AKSZ construction sends
  $\omega$ to $\int_{[T]}\mathrm{ev}^*\omega \in
  \mathcal{A}^{2,\mathrm{cl}}_S(X^T_S, s-d; \mathcal{M})$. To
  identify its underlying form, recall that
  $\mathrm{ev}^*\mathcal{M}_X$ and $\pi^*\mathcal{M}_{X^T_S}$ are
  both pulled back from $\mathcal{M}_S$. They are therefore
  equivalent, and the projection formula gives
  \begin{equation}\label{eq:ev-projection}
    \pi_*\big(\mathcal{E} \otimes \mathrm{ev}^*\mathcal{M}_X\big)
    \simeq \pi_*\mathcal{E} \otimes \mathcal{M}_{X^T_S} \qquad
    \big(\mathcal{E} \in \QCoh(X^T_S \times_S T)\big).
  \end{equation}
  With $\sigma \colon X^T_S \to S$ the structure map, the underlying
  $2$-form of $\int_{[T]}\mathrm{ev}^*\omega$ is then the composite
  \eqref{eq:CHS-aksz}, with $\pi_*\mathrm{ev}^*\mathcal{M}_X[s]
  \simeq \sigma^*\phi_*\cO_T \otimes \mathcal{M}_{X^T_S}[s]$ and
  target $\mathcal{M}_{X^T_S}[s-d]$ in place of the untwisted terms.
  The naturality of pullback and pushforward of twisted forms,
  together with the square of Lemma~\ref{lem:twisted}(5), makes the
  twisted AKSZ construction a symmetric monoidal functor, as in the
  untwisted case recalled in Section~\ref{subsec:recollection-aksz},
  \[
    \begin{tikzcd}
      \Cospan_n(\POr^{X\text{-good}}_{S,d}) \arrow{r} &
      \Span_n(\PreSymp^{\mathcal{M}}_{S,s-d}).
    \end{tikzcd}
  \]
  It remains to show that an oriented iterated cospan $\Phi$ with
  bottom vertex $T$ is sent to a twisted Lagrangian span. By
  Proposition~\ref{prop:CHS-forms}(1) and
  Proposition~\ref{prop:dag-facts}(4), the values of the diagram of
  the image span are pushforwards to $X^T_S$ of pullbacks of
  $\mathrm{ev}^*\TT_{X/S}$ and of their twisted duals. Since the
  underlying map of the twisted symplectic form of $X$ identifies
  $\TT_{X/S}$ with $\LL_{X/S} \otimes \mathcal{M}_X[s]$, it follows
  from \eqref{eq:ev-projection} that the diagram of the image span is
  the diagram $\widehat{(\sigma^*\Phi)}_{\mathcal{E}}$ of
  Section~\ref{subsec:recollection-aksz} for the base change
  $\sigma^*\Phi$ and $\mathcal{E} = \mathrm{ev}^*\TT_{X/S}$, tensored
  with $\mathcal{M}_{X^T_S}$. This diagram is non-degenerate because
  $\Phi$ is oriented and $\mathcal{E}$ is dualizable, and tensoring
  with an invertible sheaf preserves limit diagrams. Hence the
  functor restricts to a symmetric monoidal functor
  $\Or_n^{\mathrm{cpt},X\text{-good},S,d} \to \Leg_n^{S,s-d}$.
  Finally, the functors $\mathrm{orcut}$ and $(-)_B \times S$ do not
  involve forms, and composing with them gives the field theory of (5).

\end{proof}

Taking $S = B\Gm$ gives the $(\infty,n)$-categories of Theorem~\ref{thm:A}.

\begin{theorem}\label{thm:Leg}
  The symmetric monoidal $(\infty,n)$-category $\Leg_n^s$ has the
  properties stated in Theorem~\ref{thm:A}.
\end{theorem}

\begin{proof}
  By Proposition~\ref{prop:relative}(1), the objects of $\Leg_n^s$
  are the $s$-shifted contact derived Artin stacks. A $1$-morphism
  from $X$ to $Y$ is a span over $B\Gm$
  \[
    \begin{tikzcd}[column sep=small]
      & C \arrow{dl} \arrow{dr} & \\
      X \arrow{dr} & & Y \arrow{dl} \\
      & B\Gm &
    \end{tikzcd}
  \]
  together with a twisted Lagrangian structure on $C \to X
  \times_{B\Gm} Y$ for the form $-\pi_X^*\omega_X + \pi_Y^*\omega_Y$.
  By Proposition~\ref{prop:relative}(2), this is a Legendrian
  morphism $C \to \overline{X} \circledast Y$, where $\overline{X}$
  carries the negated structure, i.e., a Legendrian correspondence.
  Higher morphisms are described likewise. The tensor product and
  unit are identified in Proposition~\ref{prop:relative}(3), and
  adjoints and full dualizability are Theorem~\ref{thm:transfer}(3).

\end{proof}

\begin{remark}\label{rem:symplectification-functor}
  Pullback along $q \colon \mathrm{pt} \to B\Gm$ preserves fiber
  products, and it also preserves the twisted Lagrangian conditions,
  since these can be checked after pullback to affine points by
  Step~2 of the proof of Theorem~\ref{thm:transfer}. It therefore
  induces a commutative square of functors
  \[
    \begin{tikzcd}
      \Leg_n^s \arrow{r}{q^*} \arrow[hook]{d} &
      \Lag_n^{\mathrm{pt},s} \arrow[hook]{d} \\
      \Span_n(\PreSymp^{\mathcal{M}}_{B\Gm,s}) \arrow{r}{q^*} &
      \Span_n(\PreSymp_{\mathrm{pt},s}).
    \end{tikzcd}
  \] On objects it sends $X$ to its symplectification $(\widetilde
  X,\omega)$. On $1$-morphisms it sends a Legendrian correspondence
  to the corresponding $\Gm$-equivariant Lagrangian correspondence
  between symplectifications.
\end{remark}

\section{Derived Legendrian categories}\label{sec:Fc}

\begin{definition}\label{def:Fc}
  Let $X$ be an $s$-shifted contact derived Artin stack and $n \geq
  1$. The \bfem{derived Legendrian $(\infty,n)$-category} of $X$ is
  \[ \Fc(X) := \Leg_n^{X/B\Gm,s} = \Leg_{n+1}^{s}(B\Gm, X). \]
\end{definition}

For example, for the unit $X = B\Gm$, Theorem~\ref{thm:transfer}(2)
gives $\Fc(B\Gm) \simeq \Leg_n^{s-1}$, so the Legendrians in $B\Gm$
are the $(s-1)$-shifted contact stacks (Example~\ref{ex:legendrian}(2)).

\begin{theorem}\label{thm:Fc}
  The $(\infty,n)$-category $\Fc(X)$ has the properties (1)--(3) of
  Theorem~\ref{thm:B}.
\end{theorem}

\begin{proof}\mbox{}
  \begin{enumerate}[wide]
    \item A $1$-morphism $B\Gm \to X$ in $\Leg_{n+1}^s$ is a span over $B\Gm$
      \[
        \begin{tikzcd}[column sep=small]
          & L \arrow{dl} \arrow{dr} & \\
          B\Gm & & X
        \end{tikzcd}
      \]
      with a twisted Lagrangian structure on $L \to B\Gm
      \times_{B\Gm} X = X$, since the form on $B\Gm$ is zero by
      Lemma~\ref{lem:twisted}(6). By
      Proposition~\ref{prop:relative}(2), this is a Legendrian
      morphism $L \to X$.

    \item By Theorem~\ref{thm:transfer}(4),
      \[ \Fc(X)(L_1,L_2) \simeq \Leg_{n-1}^{L_2 \times_X
      L_1^{\mathrm{rev}}/B\Gm,\,s-1}. \]
      The underlying stack of $L_2 \times_X L_1^{\mathrm{rev}}$ is
      the derived intersection $L_2 \times_X L_1 \simeq L_1 \times_X
      L_2$, which is the apex of the composite span
      \[
        \begin{tikzcd}[column sep=small]
          & & L_2 \times_X L_1 \arrow{dl} \arrow{dr} & & \\
          & L_2 \arrow{dl} \arrow{dr} & & L_1 \arrow{dl} \arrow{dr} & \\
          B\Gm & & X & & B\Gm.
        \end{tikzcd}
      \] Its $(s-1)$-shifted twisted symplectic structure comes from
      Theorem~\ref{thm:transfer}(2), and by
      Proposition~\ref{prop:relative}(1) it is an $(s-1)$-shifted
      contact structure. Pulling back along $q$
      (Remark~\ref{rem:symplectification-functor}) shows that its
      symplectification is the composite of $\mathrm{pt} \to
      \widetilde X$ and $\widetilde X \to \mathrm{pt}$ in
      $\Lag_{n+1}^{\mathrm{pt},s}$,
      \[ \widetilde L_2 \times_{\widetilde X} \widetilde L_1, \]
      with the structure of Theorem~\ref{thm:CHS}(2).

    \item Composition of $1$-morphisms in $\Fc(X)$ is vertical
      composition of $2$-morphisms in $\Leg_{n+1}^s$, which is
      composition of $2$-fold spans. For $2$-fold spans with apexes
      $N_{01}$ and $N_{12}$ over the $1$-morphisms $L_0$, $L_1$ and
      $L_2$, the composite is
      \[
        \begin{tikzcd}[column sep=small]
          & & N_{01} \times_{L_1} N_{12} \arrow{dl} \arrow{dr} & & \\
          & N_{01} \arrow{dl} \arrow{dr} & & N_{12} \arrow{dl} \arrow{dr} & \\
          L_0 & & L_1 & & L_2.
        \end{tikzcd}
      \]
  \end{enumerate}
\end{proof}

\begin{corollary}[Legendrian intersection theorem]\label{cor:intersection}
  Let $L_1 \to X$ and $L_2 \to X$ be Legendrian morphisms into an
  $s$-shifted contact derived Artin stack. Then the derived
  intersection $L_1 \times_X L_2$ carries a canonical $(s-1)$-shifted
  contact structure, for every $s \in \mathbb{Z}$, whose
  symplectification is $\widetilde L_1 \times_{\widetilde X}
  \widetilde L_2$. If $\cL$ is trivial on an \'etale cover of $X$,
  the same holds for Legendrian morphisms in the sense of
  Definition~\ref{def:legendrian-berktav}, and the contact structure
  on $L_1 \times_X L_2$ is then a contact structure in the sense of
  Definition~\ref{def:berktav}. This recovers
  \cite[Theorem~B]{IzbudakBerktav2026}.
\end{corollary}

\begin{proof}
  The first statement is Theorem~\ref{thm:Fc}(2). If $\cL$ is trivial
  on an \'etale cover of $X$, then Legendrian morphisms in the sense
  of Definition~\ref{def:legendrian-berktav} are Legendrian in the
  sense of Definition~\ref{def:legendrian} by
  Theorem~\ref{thm:legendrian-comparison}. Moreover, the contact line
  bundle of $L_1 \times_X L_2$ is pulled back from $X$, so it is
  trivial on an \'etale cover, and Proposition~\ref{prop:comparison} applies.

\end{proof}

\begin{remark}\label{rem:triple}
  In the $(\infty,2)$-category $\Leg_2^s$, let $f_i \colon B\Gm \to
  X$ ($i = 0,1,2$) be Legendrian morphisms, with right adjoints
  $f_i^R$ given by the reversed spans. Composing spans gives
  \[ f_j^R \circ f_i \simeq L_i \times_X L_j. \]
  Whiskering the counit $f_1 \circ f_1^R \Rightarrow \mathrm{id}_X$
  of Theorem~\ref{thm:transfer}(3) by $f_2^R$ and $f_0$ gives a
  $2$-morphism in $\Leg_2^s(B\Gm,B\Gm) \simeq \Leg_1^{s-1}$
  \[
    \begin{tikzcd}
      (L_1 \times_X L_2) \circledast (L_0 \times_X L_1) \simeq f_2^R
      \circ f_1 \circ f_1^R \circ f_0 \arrow[Rightarrow]{r} & f_2^R
      \circ f_0 \simeq L_0 \times_X L_2.
    \end{tikzcd}
  \]
  Composing the spans shows that it is the correspondence
  \[
    \begin{tikzcd}[column sep=tiny]
      & L_0 \times_X L_1 \times_X L_2 \arrow{dl} \arrow{dr} & \\
      (L_1 \times_X L_2) \circledast (L_0 \times_X L_1) & & L_0 \times_X L_2.
    \end{tikzcd}
  \]
  Thus the triple intersection correspondences are the whiskered
  counits of the adjunctions $f_i \dashv f_i^R$.
\end{remark}

\begin{corollary}\label{cor:monad}
  Let $f \colon L \to X$ be a Legendrian morphism into an $s$-shifted
  contact derived Artin stack. Then the $(s-1)$-shifted contact stack
  $L \times_X L$ has the structure of an associative algebra in the
  monoidal $\infty$-category $\Leg_2^s(B\Gm,B\Gm) \simeq
  \Leg_1^{s-1}$, whose monoidal structure is the contact product.
\end{corollary}

\begin{proof}
  By Theorem~\ref{thm:transfer}(3), the $1$-morphism $f \colon B\Gm
  \to X$ of $\Leg_2^s$ has a right adjoint $f^R$,
  \[
    \begin{tikzcd}
      B\Gm \arrow[bend left=25]{r}{f} &[2em] X. \arrow[bend left=25]{l}{f^R}
    \end{tikzcd}
  \] By \cite[Theorem~4.4]{Haugseng2021}, due to Riehl and Verity,
  the adjunction extends to a functor from the walking adjunction
  into $\Leg_2^s$. Restricting it to the walking monad gives a monad
  on $B\Gm$ with underlying endomorphism $f^R \circ f \simeq L
  \times_X L$. By \cite[Remark~8.8 and Corollary~8.9]{Haugseng2021},
  monads on $B\Gm$ are the same as associative algebras in the
  monoidal $\infty$-category $\Leg_2^s(B\Gm,B\Gm)$ of endomorphisms
  of $B\Gm$ under composition.
  By Theorem~\ref{thm:transfer}(2), $\Leg_2^s(B\Gm,B\Gm) \simeq
  \Leg_1^{s-1}$, and composition of endomorphisms of $B\Gm$ is the
  fiber product over $B\Gm$, which is the contact product. The
  monoidal structures agree because the symmetric monoidal structure
  of Theorem~\ref{thm:transfer}(1) is obtained by this delooping, in
  Step~4 of the proof of Theorem~\ref{thm:transfer}.

\end{proof}

\section{The AKSZ construction}\label{sec:aksz}

Let $X$ be a derived stack over $B\Gm$ with principal $\Gm$-bundle
$\widetilde X$, and let $Y$ be a derived stack. Pullback along $q$
preserves internal Homs, because $\dSt$ is an $\infty$-topos, and
$q^*(Y \times B\Gm) \simeq Y$ and $q^*X \simeq \widetilde X$. Hence the square
\[
  \begin{tikzcd}
    \mathrm{Map}(Y,\widetilde X) \arrow{r} \arrow{d} & \mathrm{pt}
    \arrow{d}{q} \\
    X^{Y \times B\Gm}_{B\Gm} \arrow{r} & B\Gm
  \end{tikzcd}
\]
is a pullback, compatibly with the $\Gm$-actions, and descent along
$q$ identifies the internal Hom in $\dSt_{B\Gm}$ with the quotient
mapping stack,
\begin{equation}\label{eq:relative-mapping}
  X^{Y \times B\Gm}_{B\Gm} \simeq [\mathrm{Map}(Y,\widetilde X)/\Gm],
\end{equation}
where $\Gm$ acts through its action on $\widetilde X$.

With this identification, the AKSZ construction of
Theorem~\ref{thm:transfer}(5) gives the following result.

\begin{theorem}\label{thm:aksz}
  Let $X$ be an $s$-shifted contact derived Artin stack.
  \begin{enumerate}
    \item For every $n \geq 1$ there is an oriented extended
      topological field theory $\mathcal{Z}_X \colon
      \mathrm{Bord}^{\mathrm{or}}_{0,n} \to \Leg_n^s$, given by the composite
      \[
        \begin{tikzcd}[column sep=3.2em]
          \mathrm{Bord}^{\mathrm{or}}_{0,n} \arrow{r}{\mathrm{orcut}}
          \arrow[bend right=10]{rrr}[swap]{\mathcal{Z}_X} &
          \Or_n^{\mathcal{S},0} \arrow{r}{(-)_B \times B\Gm} &[1.8em]
          \Or_n^{\mathrm{cpt},X\text{-good},B\Gm,0}
          \arrow{r}{\mathrm{AKSZ}} & \Leg_n^{s}.
        \end{tikzcd}
      \]
    \item Let $Y$ be an $\cO$-compact derived stack with an
      $\cO$-orientation of degree $d$ in the sense of
      \cite[Definitions~2.1 and~2.4]{PTVV}. Suppose that
      $\mathrm{Map}(Y,\widetilde X)$ is a derived Artin stack. For
      instance, $Y$ can be the Betti stack $M_B$ of a closed oriented
      $d$-manifold $M$. Then the AKSZ construction equips
      $[\mathrm{Map}(Y,\widetilde X)/\Gm]$ with an $(s-d)$-shifted
      contact structure. Its symplectification is
      $\mathrm{Map}(Y,\widetilde X)$ with the transgressed
      $(s-d)$-shifted symplectic structure of \cite[Theorem~2.5]{PTVV}.
  \end{enumerate}
\end{theorem}

\begin{proof}\mbox{}
  \begin{enumerate}[wide]
    \item This is Theorem~\ref{thm:transfer}(5) for $S = B\Gm$. The
      $X$-goodness condition holds by Proposition~\ref{prop:dag-facts}(5).

    \item We check that $T := Y \times B\Gm$, with the orientation
      pulled back from $Y$, is an $X$-good $\cO$-compact $d$-oriented
      $B\Gm$-stack, and then apply Theorem~\ref{thm:transfer}(5).

      \bfem{Step 1.} The base change of $T$ along any $\Spec A \to
      B\Gm$ is given by the pullback square
      \[
        \begin{tikzcd}
          Y \times \Spec A \arrow{r} \arrow{d} & Y \times B\Gm \arrow{d} \\
          \Spec A \arrow{r} & B\Gm,
        \end{tikzcd}
      \]
      with the pulled-back orientation. It is perfect and
      cocontinuous because $Y$ is $\cO$-compact
      \cite[Definition~2.1]{PTVV}, and the orientation is
      non-degenerate \cite[Definition~2.4]{PTVV}. Since both
      conditions can be checked after base change to points
      (Section~\ref{subsec:recollection-aksz}), it follows that $T$
      is $\cO$-compact and $d$-oriented over $B\Gm$. By
      \eqref{eq:relative-mapping},
      \[ X^T_{B\Gm} \simeq [\mathrm{Map}(Y,\widetilde X)/\Gm], \]
      which is a derived Artin stack, so $T$ is $X$-good. For $Y =
      M_B$, the $d$-oriented $B\Gm$-stack $T$ is the image of $M$
      under the functor $(-)_B \times B\Gm$ of Theorem~\ref{thm:CHS}(5).

      \medskip
      \bfem{Step 2.} The AKSZ functor of
      Theorem~\ref{thm:transfer}(5) gives an $\mathcal{M}$-twisted
      $(s-d)$-symplectic structure on $X^T_{B\Gm}$, which is an
      $(s-d)$-shifted contact structure by
      Proposition~\ref{prop:relative}(1). To identify its
      symplectification, recall that the twisted pushforward of forms
      is compatible with base change (Step~6 of the proof of
      Theorem~\ref{thm:transfer}). Pulling back along $q$ therefore
      gives the AKSZ construction over $\mathrm{pt}$ for $\widetilde X$ and $Y$,
      \[ \int_{[Y]} \mathrm{ev}^*\omega \in
      \mathcal{A}^{2,\mathrm{cl}}\big(\mathrm{Map}(Y,\widetilde X), s-d\big). \]
      This is the transgression of \cite[Theorem~2.5]{PTVV}. \qedhere
  \end{enumerate}
\end{proof}

\begin{remark}\label{rem:aksz-recovers}
  Theorem~\ref{thm:aksz}(2) recovers
  \cite[Theorem~A]{IzbudakBerktav3}, which uses the definition of
  contact structures of \cite[Definition~2.38]{IzbudakBerktav2026},
  following Berktav. That definition needs the contact line bundle to
  be trivial on an \'etale cover, and this holds in the following
  case. Suppose that $\cL$ is trivial on an \'etale cover of $X$ and
  that $Y$ has a $\K$-point $y$. The structure map of
  $[\mathrm{Map}(Y,\widetilde X)/\Gm]$ to $B\Gm$ factors through
  evaluation at $y$,
  \[
    \begin{tikzcd}
      {[\mathrm{Map}(Y,\widetilde X)/\Gm]} \arrow{r}{\mathrm{ev}_y}
      \arrow{dr} & X \arrow{d}{f} \\
      & B\Gm,
    \end{tikzcd}
  \]
  so its contact line bundle is $\mathrm{ev}_y^*\cL$, which is
  trivial on an \'etale cover, and Proposition~\ref{prop:comparison}
  applies. Theorem~\ref{thm:aksz}(1) also describes cobordisms. An
  oriented cobordism $W$ between closed oriented manifolds $M_1$ and
  $M_2$ is sent by $\mathcal{Z}_X$ to the Legendrian correspondence
  \[
    \begin{tikzcd}[column sep=tiny]
      & {[\mathrm{Map}(W_B,\widetilde X)/\Gm]} \arrow{dl} \arrow{dr} & \\
      {[\mathrm{Map}(M_{1,B},\widetilde X)/\Gm]} & &
      {[\mathrm{Map}(M_{2,B},\widetilde X)/\Gm]}.
    \end{tikzcd}
  \]
  By functoriality, gluing of cobordisms is sent to composition of
  Legendrian correspondences.
\end{remark}

\section{Examples and applications}\label{sec:applications}

\subsection{Projectivized shifted cotangent stacks}

\begin{example}\label{ex:PT}
  Let $Y$ be a derived Artin stack and $s \in \mathbb{Z}$, and let
  $\Gm$ act on the shifted cotangent stack $T^*[s]Y := \Spec_Y
  \mathrm{Sym}(\TT_Y[-s])$ of \cite[\S2]{Calaque2019} by scaling the
  fibers. The quotient map $T^*[s]Y \to P(T^*[s]Y) := [T^*[s]Y/\Gm]$
  is a principal $\Gm$-bundle, i.e., the square
  \[
    \begin{tikzcd}
      T^*[s]Y \arrow{r} \arrow{d} & \mathrm{pt} \arrow{d}{q} \\
      P(T^*[s]Y) \arrow{r} & B\Gm
    \end{tikzcd}
  \]
  is a pullback. The canonical $s$-shifted symplectic structure
  $d_{dR}\lambda_Y$ of \cite[Theorem~2.4]{Calaque2019} is the de Rham
  differential of the Liouville form $\lambda_Y$, the tautological
  section of $\pi^*\LL_Y[s]$. The Liouville form is linear along the
  fibers, hence of weight one \cite[\S2.1.1]{Calaque2019}, and by
  Lemma~\ref{lem:weights}(2) so is $d_{dR}\lambda_Y$. By
  Definition~\ref{def:contact}, $P(T^*[s]Y)$ is an $s$-shifted
  contact derived Artin stack with symplectification $(T^*[s]Y,
  d_{dR}\lambda_Y)$. Note that the zero section is not removed, since
  the non-degeneracy in Definition~\ref{def:contact} is a condition
  on the symplectification.
\end{example}

\subsection{Smooth Legendrians and derived discriminant loci}

\begin{lemma}\label{lem:smooth-legendrian}
  Let $X$ be a smooth $\K$-scheme with a contact form $\alpha$ in the
  sense of Definition~\ref{def:berktav} (with $s = 0$) and contact
  line bundle $\cL$. Let $i \colon L \hookrightarrow X$ be a smooth
  closed subscheme with $i^*\alpha = 0$ and $\dim X = 2\dim L + 1$.
  Then $i$ carries a unique Legendrian structure, and the isotropic
  structure given by $i^*\alpha = 0$ is Legendrian in the sense of
  Berktav (Definition~\ref{def:legendrian-berktav}).
\end{lemma}

\begin{proof}
  By Proposition~\ref{prop:comparison}, $\omega := \Phi(\alpha)$ is a
  weight-one $0$-shifted symplectic structure on the smooth scheme
  $\widetilde X$. The space of closed $0$-shifted $2$-forms on the
  smooth scheme $\widetilde L$ is discrete \cite[\S1.2, property~(1)
  after Remark~1.19]{PTVV}. By naturality of $\Phi$,
  \[ \tilde\imath^*\omega = \Phi(i^*\alpha) = \Phi(0) = 0. \]
  Hence there is a unique path from $0$ to $\tilde\imath^*\omega$,
  and it lies in weight one. Since $\dim \widetilde X = 2\dim
  \widetilde L$, the induced sequence of vector bundles
  \[
    \begin{tikzcd}
      \TT_{\widetilde L} \arrow{r} & \tilde\imath^*\TT_{\widetilde X}
      \arrow{r} & \LL_{\widetilde L}
    \end{tikzcd}
  \]
  has $\TT_{\widetilde L}$ as an isotropic subbundle of half the
  rank, so it is exact. This is the Lagrangian condition, and hence
  this path is a Legendrian structure. For the last statement,
  $\mathcal{A}^1(L, i^*\cL; 0)$ is discrete, since $L$ is a smooth
  scheme, so $i^*\alpha = 0$ has a unique null-homotopy. It
  corresponds to the Legendrian structure above under
  Proposition~\ref{prop:comparison}(2), and since $X$ is a scheme,
  Theorem~\ref{thm:legendrian-comparison} applies and shows that this
  null-homotopy is Legendrian in the sense of Berktav.
\end{proof}

Jet spaces give examples of Lemma~\ref{lem:smooth-legendrian}. Let
$V$ be a smooth affine $\K$-scheme and $J^1(V) := T^*V \times
\mathbb{A}^1$ its $1$-jet space, with fiber coordinates $\xi$ on
$T^*V$ and coordinate $z$ on $\mathbb{A}^1$. It carries the contact
form $\alpha = d_{dR}z - \xi$ with trivial contact line bundle, where
$\xi$ also denotes the tautological $1$-form. For a function $f \in
\cO(V)$ let $L_f := j^1 f \colon V \to J^1(V)$ be its $1$-jet, $x
\mapsto (x, d_{dR}f(x), f(x))$. Then $(j^1f)^*\alpha = d_{dR}f -
d_{dR}f = 0$ and $\dim J^1(V) = 2\dim V + 1$, so by
Lemma~\ref{lem:smooth-legendrian} each $L_f$ carries a unique
Legendrian structure. Intersections of $1$-jets are the following
derived schemes. The \bfem{derived discriminant locus} of $f$ is the
derived intersection
\[
  \begin{tikzcd}
    \Delta\mathrm{loc}(f) \arrow{r} \arrow{d} & V \arrow{d}{j^1 0} \\
    V \arrow{r}{j^1 f} & J^1(V),
  \end{tikzcd}
\]
i.e., the derived zero locus of the section $(f, d_{dR}f)$ of $\cO_V
\oplus \Omega^1_V$. Its classical truncation is $\{f = 0,\ d_{dR}f = 0\}$.

The derived Legendrian category of $J^1(V)$ is described through
these loci as follows.

\begin{proposition}\label{prop:discriminant}
  Let $V$ be a smooth affine $\K$-scheme, $f, f_0, f_1, f_2 \in
  \cO(V)$, and $n \geq 1$. Write $J := J^1(V)$.
\begin{enumerate}\itemsep=0.1in
    \item $\Delta\mathrm{loc}(f)$ carries a canonical $(-1)$-shifted
      contact structure with trivial contact line bundle.
    \item For all $i,j$, the derived intersection $L_{f_i} \times_{J}
      L_{f_j}$ is, as a derived stack, the derived discriminant locus
      $\Delta\mathrm{loc}(f_j - f_i)$. The mapping
      $(\infty,n-1)$-category $\Fc(J)(L_{f_i}, L_{f_j})$ is the
      $(\infty,n-1)$-category of Legendrians in it, for the
      $(-1)$-shifted contact structure of Theorem~\ref{thm:Fc}(2).
    \item Composition in $\Fc(J)$ from $L_{f_0}$ via $L_{f_1}$ to
      $L_{f_2}$ is given on $1$-morphisms by fiber products over $V =
      L_{f_1}$. The triple intersection is the apex of the Legendrian
      correspondence of Remark~\ref{rem:triple},
      \[
        \begin{tikzcd}[column sep=tiny]
          & L_{f_0} \times_{J} L_{f_1} \times_{J} L_{f_2} \arrow{dl}
          \arrow{dr} & \\
          (L_{f_1} \times_{J} L_{f_2}) \circledast (L_{f_0}
          \times_{J} L_{f_1}) & & L_{f_0} \times_{J} L_{f_2}.
        \end{tikzcd}
      \]
  \end{enumerate}
\end{proposition}

\begin{proof}\mbox{}
  \begin{enumerate}[wide]
    \item The form $\alpha$ is a $0$-shifted contact form, as in
      \cite[Theorem~4.2]{kib2}, so $J^1(V)$ is $0$-shifted contact by
      Proposition~\ref{prop:comparison}. Since $L_f$ and $L_0$ are
      Legendrian, (1) follows from Corollary~\ref{cor:intersection}.

    \item The projection $J^1(V) \to V$ is the total space of the
      vector bundle $\cO_V \oplus \Omega^1_V$, and translation by
      $-j^1f_i$ is an automorphism of it over $V$. So the derived
      intersection of the graphs of $j^1f_i$ and $j^1f_j$ is the
      derived zero locus of the section
      \[ j^1f_j - j^1f_i = \big(f_j - f_i,\ d_{dR}(f_j - f_i)\big). \]
      The description of mapping categories is Theorem~\ref{thm:Fc}(2).

    \item This is Theorem~\ref{thm:Fc}(3) together with
      Remark~\ref{rem:triple}. \qedhere
  \end{enumerate}
\end{proof}

Derived discriminant loci are the local models of $(-1)$-shifted
contact derived schemes \cite[Corollary~3.9]{Izbudak_Darboux}.

\subsection{The global nilpotent cone}

\begin{remark}\label{rem:nilpotent}
  Let $C$ be a smooth connected projective curve of genus at least
  two, $G$ a semisimple group, and $\mathrm{Bun}_G(C)$ the moduli
  stack of principal $G$-bundles on $C$. By Example~\ref{ex:PT},
  $P(T^*\mathrm{Bun}_G(C))$ is a $0$-shifted contact derived Artin
  stack, and the global nilpotent cone is a candidate for a
  Legendrian in it. This cone goes back to Laumon, who defined it for
  vector bundles through the nilpotent pairs of
  \cite[D\'efinition~(0.2)]{Laumon} (see also the definition before
  the Main Theorem of \cite{GinzburgNilpotent}). In fact, for $\K =
  \mathbb{C}$ the reduced global nilpotent cone is a Lagrangian
  substack of $T^*\mathrm{Bun}_G(C)$ in the classical sense
  \cite[Main Theorem]{GinzburgNilpotent}. According to the remark
  after that theorem, this was first proved for $G = \mathrm{SL}_n$
  by Laumon, and in general by Faltings. However, we do not construct
  a $\Gm$-equivariant, weight-one Lagrangian structure on a derived
  enhancement $\mathcal{N} \to T^*\mathrm{Bun}_G(C)$ of the nilpotent
  cone. Given one, the projectivization $P\mathcal{N} \to
  P(T^*\mathrm{Bun}_G(C))$ is a Legendrian morphism
  (Definition~\ref{def:legendrian}), and Corollary~\ref{cor:monad}
  makes the $(-1)$-shifted contact stack $P\mathcal{N}
  \times_{P(T^*\mathrm{Bun}_G(C))} P\mathcal{N}$ an associative
  algebra in $\Leg_1^{-1}$. Note that
  Definition~\ref{def:legendrian-berktav} would not apply here, since
  the contact line bundle of $P(T^*\mathrm{Bun}_G(C))$ is not trivial
  on any \'etale cover (Remark~\ref{rem:locally}).
\end{remark}

\subsection{Outlook: contact surgery}
Theorem~\ref{thm:aksz} sends an oriented cobordism $W$ to the
Legendrian correspondence given by the quotient mapping stack of
$W_B$, and gluing of cobordisms to composition. In classical contact
topology, Weinstein's contact surgery is performed along isotropic
spheres, and along Legendrian spheres in the critical case. We expect
derived contact surgery along a Legendrian morphism $f \colon L \to
X$ to be given in $\Leg_n^s$ by composition with the correspondences
attached to handle attachments.

\appendix

\section{Diagrams of quasi-coherent sheaves and composition}\label{app:chs}

This appendix recalls the parts of \cite{CHS2025} that turn diagrams
of presymplectic stacks into diagrams of quasi-coherent sheaves, and
that show that Lagrangian correspondences compose. We also prove the
twisted form of the first of these, which is used in the proof of
Theorem~\ref{thm:transfer}, and we collect the facts about derived
stacks from \cite{CHS2025} that are used in the proofs.
\medskip

\paragraph{\bf A1. Iterated spans}
Let $\mathsf{L}^{n,\circ}$ be the partially ordered set with objects
$A_i$ and $B_i$ for $1 \leq i \leq n$, ordered by
\[ A_i \leq A_k,\ B_i \leq B_k \iff i \geq k, \qquad A_i \leq
B_k,\ B_i \leq A_k \iff i > k. \]
Let $\mathsf{L}^n := (\mathsf{L}^{n,\circ})^{\triangleleft}$ be
obtained by adding an initial object $-\infty$. An \bfem{$n$-fold
span} in an $\infty$-category $\mathcal{C}$ is a functor
$\mathsf{L}^n \to \mathcal{C}$, and an \bfem{$n$-uple span} is a
functor $\mathsf{Sp}^n \to \mathcal{C}$, where $\mathsf{Sp}^n :=
(\mathsf{L}^1)^{\times n}$ \cite[Definitions~A.2.1
and~A.2.3]{CHS2025}. In both cases the value at $-\infty$ is the
\bfem{apex}. The two notions are related by a functor $j_n \colon
\mathsf{Sp}^n \to \mathsf{L}^n$ identifying $n$-fold spans with the
reduced $n$-uple spans \cite[Definition~A.2.7, Lemma~A.2.10]{CHS2025}.

For an $\infty$-category $\mathcal{C}$ with pullbacks, the
$(\infty,n)$-category $\Span_n(\mathcal{C})$ is built from the
partially ordered sets $\mathsf{S}^{\mathbf{k}} \supseteq
\mathsf{L}^{\mathbf{k}}$ of \cite[Definition~A.2.11]{CHS2025}, for
$\mathbf{k} \in \Delta^n := \Delta^{\times n}$ with $\Delta$ the
simplex category \cite[\S1.5]{CHS2025}, which index composable grids
of iterated spans. A functor $\mathsf{S}^{\mathbf{k}} \to
\mathcal{C}$ is \bfem{cartesian} if it is a right Kan extension of
its restriction to $\mathsf{L}^{\mathbf{k}}$ \cite[Definitions~A.2.12
and~A.2.13]{CHS2025}, so that its values are the iterated fiber
products that compose the spans. The cartesian functors form an
$\infty$-category $\mathrm{SPAN}^{+}_{n}(\mathcal{C})$ over
$\Delta^{n,\mathrm{op}}$, and we use the following facts about it.

\begin{theorem}[{\cite{CHS2025}}]\label{thm:spans}
  Let $\mathcal{C}$ be an $\infty$-category with pullbacks.
\begin{enumerate}\itemsep=0.1in
    \item The projection $\mathrm{SPAN}^{+}_{n}(\mathcal{C}) \to
      \Delta^{n,\mathrm{op}}$ is a cocartesian fibration, and the
      corresponding functor is an $n$-uple category object in
      $\infty$-categories. Its underlying $n$-fold Segal space
      $\Span_n(\mathcal{C})$ is complete.
    \item For objects $x, y \in \mathcal{C}$ there is a natural
      equivalence $\Span_n(\mathcal{C})(x,y) \simeq
      \Span_{n-1}(\mathcal{C}_{/x,y})$.
    \item Suppose that $(\mathcal{C}_k, c_k)_{k \geq 0}$ are pointed
      $\infty$-categories with pullbacks, with equivalences
      $\mathcal{C}_k \simeq (\mathcal{C}_{k+1})_{/c_{k+1}, c_{k+1}}$
      under which $c_k$ corresponds to the degenerate span of
      $c_{k+1}$. Then $\Span_n(\mathcal{C}_0)$ is symmetric monoidal
      with unit $c_0$.
    \item A functor $\mathcal{C} \to \mathcal{D}$ that preserves
      pullbacks induces a functor $\Span_n(\mathcal{C}) \to
      \Span_n(\mathcal{D})$, compatible with the equivalences of (2).
    \item $\Span_n(\mathcal{C})$ has adjoints for $1$-morphisms, and
      the adjoint of a span is the reversed span.
  \end{enumerate}
\end{theorem}

Parts (1)--(4) are \cite[Theorem~A.2.14, Propositions~A.2.18
and~A.2.19, and Remarks~A.2.21 and~A.2.22]{CHS2025}, which recall
results of Haugseng. Part (5) is the result of Haugseng quoted at the
start of the proof of \cite[Proposition~2.11.1]{CHS2025}.
\medskip

\paragraph{\bf A.2 Diagrams of quasi-coherent sheaves}
Let $\mathcal{Q} \to \dSt^{\mathrm{op}}$ be the cocartesian fibration
associated to the functor $\QCoh(-)$ with pullbacks
\cite[Definition~B.5.5]{CHS2025}. Its objects are pairs $(X,
\mathcal{E})$ with $\mathcal{E} \in \QCoh(X)$, and its cocartesian
morphisms over $f \colon X \to X'$ are the maps $(X', \mathcal{E}')
\to (X, f^*\mathcal{E}')$. Relative forms are recorded in
$\mathcal{Q}$ by the following functor.

\begin{proposition}[{\cite[Lemma~B.12.4]{CHS2025}}]\label{prop:forms-functor}
  There is a functor
  \[
    \begin{tikzcd}
      \big((\dSt^{\mathrm{geom}}_S)_{/\mathcal{A}^p_S(s)}\big)^{\mathrm{op}}
      \arrow{r} & \mathcal{Q}^{\Delta^1}
    \end{tikzcd}
  \]
  that sends a geometric $S$-stack $X$ with a relative $p$-form
  $\omega$ to the corresponding map $\cO_X \to \otimes^p\LL_{X/S}[s]$.
\end{proposition}

It is constructed from Proposition~\ref{prop:CHS-forms}(2), which
writes $\mathcal{A}^p_S(-,s)$ as the composite of the functor $X
\mapsto (X, \Lambda^p\LL_{X/S}[s])$ to $\mathcal{Q}$ with
$\mathrm{Map}_{\mathcal{Q}}(\cO, -)$.

For an $\infty$-category $\mathcal{C}$, let
$\mathrm{Tw}^r(\mathcal{C})$ be its twisted arrow $\infty$-category,
and recall that the functor $\mathrm{Tw}^r$ has a left adjoint
$\mathrm{Tw}^r_!$ \cite[Corollary~A.4.4]{CHS2025}. For a partially
ordered set such as $\mathsf{Sp}^n$, the partially ordered set
$\widetilde{\mathsf{Sp}}{}^n := \mathrm{Tw}^r_!(\mathsf{Sp}^n)$ has
objects $I$ and $I^\vee$ for $I \in \mathsf{Sp}^n$, ordered as in
\cite[Remark~2.6.13]{CHS2025}. More precisely, we have $I \leq J$,
and $J^\vee \leq I^\vee$, if and only if $I \leq J$ in
$\mathsf{Sp}^n$. Moreover, $I \leq J^\vee$ if and only if $I$ and $J$
have a common upper bound in $\mathsf{Sp}^n$, while no relation
$I^\vee \leq J$ holds. For instance, for $n = 1$ it is
\[
  \begin{tikzcd}[column sep=small]
    & -\infty \arrow{dl} \arrow{dr} & \\
    A_1 \arrow{d} & & B_1 \arrow{d} \\
    A_1^\vee \arrow{dr} & & B_1^\vee \arrow{dl} \\
    & -\infty^\vee. &
  \end{tikzcd}
\]
We write $\widetilde{\mathsf{L}}{}^n := \mathrm{Tw}^r_!(\mathsf{L}^n)$.

Diagrams of presymplectic stacks give diagrams of quasi-coherent
sheaves indexed by these partially ordered sets, as follows.

\begin{proposition}[{\cite{CHS2025}}]\label{prop:qcoh-diagram}
  Let $p \colon K^{\triangleleft} \to \PreSymp_{S,s}$ be a diagram
  sending $-\infty$ to $X$. For $k \in K$ write $X_k := p(k)$ and
  $f_k \colon X \to X_k$ for the value of $p$ on $-\infty \to k$.
  \begin{enumerate}
    \item There is a diagram $\tilde p \colon
      (K^{\triangleleft})^{\mathrm{op}} \to \mathrm{Tw}^r(\QCoh(X))$
      that sends $k$ to the map $f_k^*\TT_{X_k/S} \to
      f_k^*\LL_{X_k/S}[s]$ induced by the $2$-form of $X_k$.
    \item By adjunction, $\tilde p$ corresponds to a diagram $\hat p
      \colon \mathrm{Tw}^r_!(K^{\triangleleft})^{\mathrm{op}} \to
      \QCoh(X)$ with $\hat p(k) = f_k^*\TT_{X_k/S}$ and $\hat
      p(k^\vee) = f_k^*\LL_{X_k/S}[s]$. In particular, an $n$-uple
      span with apex $X$ gives a diagram $\widetilde{\mathsf{Sp}}{}^n
      \to \QCoh(X)$, and an $n$-fold span gives a diagram
      $\widetilde{\mathsf{L}}{}^n \to \QCoh(X)$.
    \item For $\sigma \colon S' \to S$ the diagram of $\sigma^*p$ is
      the pullback of the diagram of $p$.
  \end{enumerate}
\end{proposition}

Parts (1)--(3) are \cite[Proposition~2.6.4, Corollary~2.6.5,
Remark~2.6.6 and Corollary~2.6.8]{CHS2025}. The proof of (1) composes
$p$ with the functor of Proposition~\ref{prop:forms-functor}, pushes
the result forward along the cocartesian fibration $\mathcal{Q} \to
\dSt^{\mathrm{op}}$ to the fiber over $X$, and passes to adjoint
maps. The last step uses that the cotangent complexes $\LL_{X_k/S}$
are dualizable.

A diagram indexed by $\widetilde{\mathsf{Sp}}{}^n$ or
$\widetilde{\mathsf{L}}{}^n$ is \bfem{non-degenerate} if it is a
limit diagram \cite[Definition~2.6.10]{CHS2025}, and an $n$-uple or
$n$-fold span in $\PreSymp_{S,s}$ is \bfem{Lagrangian} if its diagram
of Proposition~\ref{prop:qcoh-diagram}(2) is non-degenerate
\cite[Definition~2.6.11]{CHS2025}. This is compatible with $j_n$,
since an $n$-fold span is Lagrangian if and only if the corresponding
reduced $n$-uple span is \cite[Proposition~2.6.14, Corollary~2.6.16]{CHS2025}.

For $n = 1$ the diagram of a span $X \leftarrow L \to Y$ is
\[
  \begin{tikzcd}[column sep=small]
    & \TT_{L/S} \arrow{dl} \arrow{dr} & \\
    f^*\TT_{X/S} \arrow{d} & & g^*\TT_{Y/S} \arrow{d} \\
    f^*\LL_{X/S}[s] \arrow{dr} & & g^*\LL_{Y/S}[s] \arrow{dl} \\
    & \LL_{L/S}[s]. &
  \end{tikzcd}
\]
It is non-degenerate if and only if the square
\eqref{eq:lagrangian-correspondence} is cartesian
\cite[Remark~2.6.1]{CHS2025}. The following criteria describe
non-degeneracy through a smaller diagram.

\begin{proposition}[{\cite{CHS2025}}]\label{prop:nondeg-criteria}
  Let $\mathcal{C}$ be a stable $\infty$-category.
  \begin{enumerate}
    \item Let $\Phi \colon \widetilde{\mathsf{Sp}}{}^n \to
      \mathcal{C}$ be a diagram, $n > 1$, whose restrictions to the
      boundary faces are non-degenerate. Let $\xi \colon
      \mathsf{Sp}^{n,\mathrm{op}} \hookrightarrow
      (\mathsf{Sp}^{1,\mathrm{op}})^{\triangleleft} \times
      \mathsf{Sp}^{n-1,\mathrm{op}}$ be the inclusion and $\Psi :=
      \xi_*(\Phi|_{\mathsf{Sp}^{n,\mathrm{op}}})$ the right Kan
      extension. Then $\Phi$ is non-degenerate if and only if
      $\Phi(-\infty)$ is the limit of $\Psi|_{\{-\infty\} \times
      \mathsf{Sp}^{n-1,\mathrm{op}}}$.
    \item A diagram $F \colon (\mathsf{L}^n)^{\triangleright} \to
      \mathcal{C}$ is a limit diagram if and only if the square
      \[
        \begin{tikzcd}
          F(-\infty) \arrow{r} \arrow{d} & F(A_n) \arrow{d} \\
          F(B_n) \arrow{r} & M_{n-1}
        \end{tikzcd}
      \]
      is cartesian, where $M_0 := F(\infty)$ and $M_i := F(A_i)
      \times_{M_{i-1}} F(B_i)$.
    \item Let $\Phi \colon \widetilde{\mathsf{Sp}}{}^n \to
      \mathcal{C}$ be a diagram whose restrictions to the boundary
      faces are non-degenerate. Then $\Phi$ is non-degenerate if and
      only if $\Phi(-\infty)$ is the limit of the restriction of
      $\Phi$ to $\mathsf{Sp}^{n,\mathrm{op}}$, and if and only if
      $\Phi$ is a right Kan extension of this restriction.
  \end{enumerate}
\end{proposition}

Parts (1)--(3) are \cite[Corollary~2.8.3 and Propositions~2.8.4
and~2.8.2]{CHS2025}.
\medskip

\paragraph{\bf A3. Composition}
To state that Lagrangian correspondences compose, we extend the
Lagrangian condition to grids of spans. A map
$\mathsf{L}^{\mathbf{k}} \to \PreSymp_{S,s}$ is \bfem{Lagrangian} if
its restrictions to all the cubes $\mathsf{Sp}^m \simeq
\mathsf{S}^{\mathbf{i}} \to \mathsf{L}^{\mathbf{k}}$, for inert maps
$\mathbf{i} \to \mathbf{k}$ with each $i_t \in \{0,1\}$, are $m$-fold
Lagrangian correspondences, and a map $\mathsf{S}^{\mathbf{k}} \to
\PreSymp_{S,s}$ is then Lagrangian if it is cartesian and its
restriction to $\mathsf{L}^{\mathbf{k}}$ is Lagrangian
\cite[Definitions~2.9.2 and~2.9.3]{CHS2025}.

The Lagrangian maps span a full subcategory
$\mathrm{LAG}^{+,S,s}_{n}$ of $\mathrm{SPAN}^{+}_{n}(\PreSymp_{S,s})$
\cite[Definition~2.9.4]{CHS2025}, and the statement that Lagrangian
correspondences compose means that the cocartesian pushforwards of
$\mathrm{SPAN}^{+}_{n}(\PreSymp_{S,s})$ preserve $\mathrm{LAG}^{+,S,s}_{n}$.

By the active-inert factorization this reduces to the degeneracy $s^0
\colon [1] \to [0]$ and the inner face $d^1 \colon [1] \to [2]$ in
one coordinate \cite[proof of Theorem~2.9.5]{CHS2025}, and the
following two results handle these cases. For them we write
$\mathsf{S}^{2,1,\ldots,1} := \mathsf{S}^2 \times \mathsf{Sp}^{n-1}$,
with the subposet $\mathsf{L}^{2,1,\ldots,1}$, and
$\widetilde{\mathsf{S}}{}^{2,1,\ldots,1} :=
\mathrm{Tw}^r_!(\mathsf{S}^{2,1,\ldots,1})$.

\begin{proposition}[{\cite{CHS2025}}]\label{prop:composition-criteria}
  Let $\mathcal{C}$ be a stable $\infty$-category.
  \begin{enumerate}
    \item For $1 \leq j \leq n$, let $\pi_j \colon \mathsf{Sp}^n \to
      \mathsf{Sp}^{n-1}$ forget the $j$-th factor. If $F \colon
      \widetilde{\mathsf{Sp}}{}^{n-1} \to \mathcal{C}$ is a limit
      diagram, then so is $F \circ \mathrm{Tw}^r_!(\pi_j)$.
    \item Let $\Phi \colon \widetilde{\mathsf{S}}{}^{2,1,\ldots,1}
      \to \mathcal{C}$ be a diagram with the following properties.
      \begin{enumerate}
        \item For every $X \in \mathsf{Sp}^{n-1}$ the square
          \[
            \begin{tikzcd}
              \Phi(((1,1),X)^\vee) \arrow{r} \arrow{d} &
              \Phi(((0,1),X)^\vee) \arrow{d} \\
              \Phi(((1,2),X)^\vee) \arrow{r} & \Phi(((0,2),X)^\vee)
            \end{tikzcd}
          \]
          is a pullback square.
        \item The restriction of $\Phi$ along every inclusion
          $\widetilde{\mathsf{Sp}}{}^n \to
          \widetilde{\mathsf{S}}{}^{2,1,\ldots,1}$ induced by an
          inclusion $\mathsf{Sp}^n \hookrightarrow
          \mathsf{L}^{2,1,\ldots,1}$ is non-degenerate.
        \item For every $X \in \mathsf{S}^{2,1,\ldots,1}$ the object
          $\Phi(X)$ is the limit of the restriction of $\Phi$ to
          $\mathsf{L}^{2,1,\ldots,1}_{X/}$.
      \end{enumerate}
      Then the restriction of $\Phi$ along the inclusion
      $\widetilde{\mathsf{Sp}}{}^n \hookrightarrow
      \widetilde{\mathsf{S}}{}^{2,1,\ldots,1}$ induced by
      \[
        \begin{tikzcd}
          \mathsf{Sp}^n \arrow{r}{(d^1, \mathrm{id}, \ldots,
          \mathrm{id})} &[3.5em] \mathsf{S}^{2,1,\ldots,1}
        \end{tikzcd}
      \]
      is non-degenerate.
  \end{enumerate}
\end{proposition}

Parts (1) and (2) are \cite[Propositions~2.9.9 and~2.9.13, and
Corollary~2.9.15]{CHS2025}. For a Lagrangian map $F \colon
\mathsf{S}^{2,1,\ldots,1} \to \PreSymp_{S,s}$,
\cite[Corollary~2.9.16]{CHS2025} applies (2) to the diagram $\Phi$ of
$F$ given by Proposition~\ref{prop:qcoh-diagram}, and the hypotheses
hold for the following reasons. Hypothesis (b) is the Lagrangian
condition on $F$, and (c) holds because $F$ is cartesian and the
relative tangent complex takes finite limits of Artin $S$-stacks to
limits \cite[Proposition~B.10.14]{CHS2025}. Hypothesis (a) follows
from (c), since the square in (a) is the dual of a pullback square of
tangent complexes, shifted by $s$, and hence a pushout square, which
in a stable $\infty$-category is also a pullback square.
\medskip

\paragraph{\bf A.4 Twisted diagrams}
Let $S \to B\Gm$ be a derived stack over $B\Gm$ and $\mathcal{M}_S$
the pullback of $\mathcal{M}$, as in Section~\ref{sec:twisted}, and
denote by $\mathcal{M}_X$ the pullback of $\mathcal{M}_S$ to an
$S$-stack $X$. The diagrams of Proposition~\ref{prop:qcoh-diagram}
have the following twisted form.

\begin{proposition}\label{prop:twisted-qcoh-diagram}
  Let $p \colon K^{\triangleleft} \to \PreSymp^{\mathcal{M}}_{S,s}$
  be a diagram sending $-\infty$ to $X$, with $X_k := p(k)$ and $f_k
  \colon X \to X_k$ as in Proposition~\ref{prop:qcoh-diagram}.
\begin{enumerate}\itemsep=0.1in
    \item There is a diagram $\tilde p \colon
      (K^{\triangleleft})^{\mathrm{op}} \to \mathrm{Tw}^r(\QCoh(X))$
      that sends $k$ to the map
      \[
        \begin{tikzcd}
          f_k^*\TT_{X_k/S} \arrow{r} & f_k^*\LL_{X_k/S} \otimes \mathcal{M}_X[s]
        \end{tikzcd}
      \]
      induced by the twisted $2$-form of $X_k$, and a corresponding
      diagram $$\hat p \colon
      \mathrm{Tw}^r_!(K^{\triangleleft})^{\mathrm{op}} \to \QCoh(X).$$
    \item For $\sigma \colon S' \to S$ over $B\Gm$, the diagram of
      $\sigma^*p$ is the pullback of the diagram of $p$.
  \end{enumerate}
\end{proposition}
\vspace{1in}

\begin{proof}\mbox{}
  \begin{enumerate}[wide]
    \item By Lemma~\ref{lem:twisted}(4), there is a functor from
      twisted $p$-forms on geometric $S$-stacks to
      $\mathcal{Q}^{\Delta^1}$ sending $(X_k, \omega_k)$ to the map
      \[
        \begin{tikzcd}
          \cO_{X_k} \arrow{r}{\omega_k} & (\otimes^2\LL_{X_k/S})
          \otimes \mathcal{M}_{X_k}[s].
        \end{tikzcd}
      \]
      If $f \colon K^{\triangleleft} \to \dSt$ is the underlying
      diagram of stacks, composing with $p^{\mathrm{op}}$ gives a
      functor $(K^{\triangleleft})^{\mathrm{op}} \to
      \mathcal{Q}^{\Delta^1}$ lying over $f^{\mathrm{op}}$. Pushing
      forward along the cocartesian morphisms of $\mathcal{Q} \to
      \dSt^{\mathrm{op}}$ lying over the maps $f_k \colon X \to X_k$
      gives a functor $(K^{\triangleleft})^{\mathrm{op}} \to
      \QCoh(X)^{\Delta^1}$ that sends $k$ to the map
      \[
        \begin{tikzcd}
          \cO_X \simeq f_k^*\cO_{X_k} \arrow{r}{f_k^*\omega_k} &
          f_k^*\big((\otimes^2\LL_{X_k/S}) \otimes
          \mathcal{M}_{X_k}\big)[s] \simeq (\otimes^2
          f_k^*\LL_{X_k/S}) \otimes \mathcal{M}_X[s].
        \end{tikzcd}
      \]
      Here $f_k^*\mathcal{M}_{X_k} \simeq \mathcal{M}_X$, since both
      are pulled back from $\mathcal{M}_S$. It remains to pass to
      adjoint maps. The object $f_k^*\LL_{X_k/S}$ is dualizable by
      Proposition~\ref{prop:CHS-forms}(2) and $\mathcal{M}_X$ is
      invertible, and a map $\cO_X \to \mathcal{E} \otimes
      \mathcal{F}$ with $\mathcal{E}$ dualizable is adjoint to a map
      $\mathcal{E}^\vee \to \mathcal{F}$. The last step of the proof
      of \cite[Proposition~2.6.4]{CHS2025} performs this passage to
      adjoint maps functorially, with values in
      $\mathrm{Tw}^r(\QCoh(X))$, and it only uses that $\mathcal{E}$
      is dualizable. With $\mathcal{E} = f_k^*\LL_{X_k/S}$ and
      $\mathcal{F} = f_k^*\LL_{X_k/S} \otimes \mathcal{M}_X[s]$ it
      therefore gives $\tilde p$, and $\hat p$ is its adjoint under
      $\mathrm{Tw}^r_! \dashv \mathrm{Tw}^r$.

    \item The construction in (1) is compatible with pullback along
      $\sigma$, by Lemma~\ref{lem:twisted}(1),
      Proposition~\ref{prop:CHS-forms}(1) and $\sigma^*\mathcal{M}_S
      \simeq \mathcal{M}_{S'}$. \qedhere
  \end{enumerate}

\end{proof}

\medskip

\paragraph{\bf A.5 Facts about derived stacks}\label{subsec:dag-facts}
The proofs above use the following facts from \cite{CHS2025}.

\begin{proposition}[{\cite{CHS2025}}]\label{prop:dag-facts}
  Let $S$ be a derived stack.
\begin{enumerate}\itemsep=0.1in
    \item Let $X$ be an $S$-stack, and for a point $p \colon \Spec A
      \to S$ write $\bar p \colon X_p \to X$ for the base change of
      $p$. A morphism in $\QCoh(X)$ is an equivalence if and only if
      its pullback along $\bar p$ is an equivalence for every $p$. A
      finite diagram in $\QCoh(X)$ is a limit diagram if and only if
      its pullback along $\bar p$ is a limit diagram for every $p$.
    \item A finite limit of Artin $S$-stacks, computed in $\dSt_S$,
      is an Artin $S$-stack.
    \item Let $X$ be the limit in $\dSt_S$ of a finite weakly
      contractible diagram of Artin $S$-stacks $X_i$, for instance a
      pullback square, and let $\pi_i \colon X \to X_i$ be the
      projections. Then the objects $\pi_i^*\LL_{X_i/S}$ form a
      colimit diagram in $\QCoh(X)$ with colimit $\LL_{X/S}$. Dually,
      the objects $\pi_i^*\TT_{X_i/S}$ form a limit diagram with
      limit $\TT_{X/S}$.
    \item Let $X$ be an Artin $S$-stack and $T \to S$ an
      $\cO$-compact morphism, and write $\mathrm{ev} \colon T
      \times_S X^T_S \to X$ and $\pi \colon T \times_S X^T_S \to
      X^T_S$ for the evaluation and projection maps. Then
      $(\pi_*\mathrm{ev}^*\TT_{X/S})^\vee$ is a relative cotangent
      complex of $X^T_S$ over $S$.
    \item Let $K$ be a finite space and $X$ an Artin $S$-stack. Then
      $X^{K_B \times S}_S \simeq \lim_K X$ in $\dSt_S$, so $X^{K_B
      \times S}_S$ is an Artin $S$-stack.
  \end{enumerate}
\end{proposition}

Parts (1)--(5) are \cite[Corollaries~B.5.9 and~B.5.11]{CHS2025},
\cite[Proposition~B.7.5]{CHS2025},
\cite[Proposition~B.10.14]{CHS2025},
\cite[Proposition~B.10.21]{CHS2025}, and \cite[Remark~3.3.6]{CHS2025}.

\newpage

\bibliographystyle{amsalpha}
\bibliography{refss}

\end{document}